\documentclass[a4paper,11pt]{article}
\usepackage{stmaryrd}
\usepackage{amsfonts}
\usepackage{bbm}
\usepackage{amscd}
\usepackage{mathrsfs}
\usepackage{latexsym,amssymb,amsmath,amscd,amscd,amsthm,amsxtra,xypic}
\usepackage[dvips]{graphicx}
\usepackage[utf8]{inputenc}
\usepackage[T1]{fontenc}
\usepackage{enumerate}
\usepackage{lmodern}
\usepackage{amssymb}
\usepackage[all]{xy}
\usepackage{ifpdf}
\ifpdf
\usepackage[colorlinks=true,linkcolor=blue,final,hyperindex,citecolor=red]{hyperref}
\else
\usepackage[colorlinks,final,hyperindex,hypertex]{hyperref}
\fi

\usepackage{nicefrac,mathtools,enumitem}
\usepackage{microtype}
\usepackage{indentfirst}
\usepackage{cite}

\usepackage{geometry}

\allowdisplaybreaks
\newtheorem{defn}{Definition}[section]
\newtheorem{thm}[defn]{Theorem}
\newtheorem{lem}[defn]{Lemma}
\newtheorem{prop}[defn]{Proposition}
\newtheorem{cor}[defn]{Corollary}
\newtheorem{ex}[defn]{Example}
\newtheorem{re}[defn]{Remark}

\numberwithin{equation}{section}

\newcommand{\emptycomment}[1]{}
\begin{document}

\title{\sc Bialgebras, Manin triples of Malcev-Poisson algebras and post-Malcev-Poisson algebras}
\author{Fattoum Harrathi$^2$ \footnote{harrathifattoum285@gmail.com}\ ,  Sami Mabrouk$^2$ \footnote{mabrouksami00@yahoo.fr}\ ,\\ 
Nasser Nawel$^1$ \footnote{nassernawel01@gmail.com}\ , Sergei Silvestrov$^3$ \footnote{sergei.silvestrov@mdu.se, (corresponding author)}}
\date{{\small{$^{1}$   Faculty of Sciences, University of Sfax,   BP
1171, 3000 Sfax, Tunisia}\\
 {\small {$^2$ Faculty of Sciences, University of Gafsa, BP 2100, Gafsa, Tunisia }\\
{\small {$^3 $ M\"{a}lardalen University,
Division of Mathematics and Physics,} \\
{\small
School of Education, Culture and Communication,} \\
\nopagebreak
{\small \hspace{0.5 cm} Box 883, 72123 V\"{a}steras, Sweden}}}}} 

\maketitle

%\renewcommand{\thefootnote}{\fnsymbol{footnote}}
%\footnote[0]{ Corresponding author \ (S. Mabrouk): mabrouk.sami00@yahoo.fr}

\abstract{A Malcev-Poisson algebra is a Malcev algebra together with a commutative
associative algebra structure related by a Leibniz rule. In this paper, we introduce the  notion of Malcev-Poisson bialgebra as an analogue of a Malcev bialgebra and establish the equivalence between matched pairs, Manin triples and Malcev-Poisson bialgebras. Moreover, we introduce a new algebraic structure, called post-Malcev-Poisson algebras. Post-Malcev-Poisson algebras can be viewed as the underlying algebraic
structures of weighted relative Rota-Baxter operators on Malcev-Poisson algebras.} 

\vspace{0.3cm}
\noindent {\bf Keywords  and phrases:} associative algebra, Malcev algebra, Malcev bialgebra,  Malcev-Poisson bialgebra, post-Malcev-Poisson algebra, representation, matched pair, Manin triple, relative Rota-Baxter operator, Yang-Baxter equation. 

\vspace{0.3cm}
\noindent {\bf MSC2020:} 17A30, 16D10, 	17B63, 17D10, 17B38

\tableofcontents
\section{Introduction}
\textbf{Malcev-Poisson  algebras.}
The notion of  Malcev algebras was introduced by Malcev \cite{malcev}, who called these objects Moufang-Lie algebras.
  %A  Malcev algebra is a non-associative algebra $A$ with an anti-symmetric multiplication $[\cdot,\cdot]$ that satisfies the Malcev identity for $x,y,z \in A$,
%\begin{equation}
%J(x,y,[x,z]) = [J(x,y,z),x],
%\end{equation}
%where $J(x,y,z) = [[x,y],z] + [[z,x],y] + [[y,z],x]$ is the Jacobian.
 %In particular, Lie algebras are examples of Malcev algebras.
 Malcev algebras play an important role in the geometry of smooth loops.   Just as the tangent
 algebra of a Lie group is a Lie algebra, the tangent algebra of a locally analytic Moufang loop is a Malcev algebra
 \cite{Kerdman,kuzmin2,malcev,nagy,sabinin}.  The reader is referred to \cite{gt,myung,okubo} for discussions about the relationships between Malcev algebras, exceptional Lie algebras and physics.
Closely related to Malcev algebras are alternative algebras.  An alternative algebra is an algebra whose associator is an alternating function.  In particular, all associative algebras are alternative, but there are plenty of non-associative alternative algebras, such as the octonions.  Roughly speaking, alternative algebras are related to Malcev algebras as associative algebras are related to Lie algebras.  Indeed, as Malcev observed in \cite{malcev}, every alternative algebra $A$ is Malcev-admissible.  There are many Malcev-admissible algebras that are not alternative; see, e.g., \cite{myung}.  

Poisson algebras were originally coming from the study of the Hamiltonian mechanics and then appeared in many fields in mathematics and mathematical physics. In \cite{Shestakov}, Shestakov was the first who generalized the
notion of a Poisson algebra to that of a Malcev-Poisson algebra and studied this
class of algebras and some related problems. Malcev–Poisson algebra  is defined to be a linear space endowed
with a Malcev bracket and a commutative associative structure which are satisfying the Leibniz rule. Recently Malcev-Poisson-Jordan algebras were presented in \cite{Said} as a natural reformulation of Poisson algebras endowed with a Malcev structure
and Jordan conditions (instead of the Lie structure and associativity, respectively) related by a Leibniz identity. Furthermore, in the referred paper
the authors presented Malcev-Poisson-Jordan structures on some classes
of Malcev algebras, also study the concept of pseudo-Euclidean Malcev-Poisson-Jordan algebras and nilpotent pseudo-Euclidean Malcev-Poisson-Jordan algebras are defined.

\textbf{Bialgebras.} A bialgebra structure is obtained from a corresponding set of comultiplication together with a set of compatibility conditions between the multiplication and the comultiplication. For example, for a finite-dimensional linear space $V$ with a given algebraic structure, this can be achieved by equipping the dual space $V^*$ with the same algebraic
structure and a set of compatibility conditions between the structures on $V$ and those on $V^*$.
One important reason for the usefulness of the Lie bialgebra is that it has a coboundary theory, which
leads to the construction of Lie bialgebras from solutions of the classical Yang-Baxter equation.
The origin of the Yang-Baxter-equations is purely physics. They were first introduced by Baxter, and Yang in \cite{R.J1,R.J2, yang}.
The notion of mock-Lie bialgebra was introduced in \cite{BCHM} as an equivalent to a Manin triple of mock-Lie algebras. 
The notion of Malcev bialgebra was introduced by Vershinin in \cite{versh}  as an analogue of a Lie bialgebra
(also see \cite{goncharov}).

\textbf{Yang–Baxter Equation and relative Rota-Baxter algebras.}
A class of Malcev bialgebras (coboundary
cases) are obtained from the solutions of an algebraic equation in a Malcev
algebra, which is an analogue of the classical Yang-Baxter equation (CYBE)
in a Lie algebra (\cite{goncharov}). It is called Malcev Yang-Baxter equation (MYBE)
for convenience. The CYBE arose from the study of inverse scattering theory in 1980s. Later it was recognized as the "semi-classical
limit" of the quantum Yang-Baxter equation which was encountered by C. N. Yang in the
computation of the eigenfunctions of a one-dimensional fermion gas with delta function
interactions \cite{yang} and by R. J. Baxter in the solution of the eight vertex model in statistical
mechanics \cite{baxter}. The study of the CYBE is also related to classical integrable systems and
quantum groups (see \cite{Chari} and the references therein).
An important approach in the study of the CYBE was through the interpretation of its
tensor form in various operator forms which proved to be effective in providing solutions of
the CYBE, in addition to the well-known work of Belavin and Drinfeld \cite{BelavinDrinfeld}.

In the case of Lie algebras, it was Semenov-Tian-Shansky who first introduced
certain operator form of the CYBE (\cite{Semenov}). Later Kupershmidt introduced
a notion of $\mathcal{O}$-operator of a Lie algebra as a generalization of (the operator
form of) the CYBE in a Lie algebra \cite{K2}. Nevertheless, the $\mathcal{O}$-operators of
Lie algebras play more interesting roles in the study of the CYBE. In fact, a
skew-symmetric solution of the CYBE is exactly a special $\mathcal{O}$-operator (associated
to the coadjoint representation) and more importantly, there are a kind
of algebraic structures behind the $\mathcal{O}$-operators of Lie algebras and the related
CYBE, namely, pre-Lie algebras, in the following sense: the $\mathcal{O}$-operators of
Lie algebras provide a direct relationship between Lie algebras and pre-Lie
algebras and in the invertible cases, they provide a necessary and sufficient
condition for the existence of a compatible pre-Lie algebra structure on a Lie
algebra, and as an immediate consequence, there are some solutions of the CYBE
in certain Lie algebras obtained from pre-Lie algebras \cite{BaiNi}.

Post-Lie algebras have been introduced by Vallette in 2007 \cite{Vallette}
in connection with the homology of partition posets and the study of Koszul
operads. Also, J. L. Loday studied pre-Lie algebras and post-Lie algebras
within the context of algebraic operad triples; see for more details \cite{Loday01}. In the last decade, many works \cite{Burde16,Fard,sl2} have appeared
devoted to post-Lie algebra structures, motivated by the importance of
post-Lie algebra in geometry and in connection with generalized Lie algebra
derivations. 
Recently, Post-Malcev algebras is introduced in \cite{postMalcev}, which is  a generalization of  pre-Malcev study by S. Madariaga in  \cite{Sarra} (see also \cite{Fattoum}), 
consist of a linear space $A$ equipped with a Malcev bracket $[\cdot ,\cdot ]$
and a binary operation $\triangleright $ satisfying some compatibilties.
 
%\begin{align}
%\label{postmalcevalgebras1}
%&\{x, z\}\rhd [y, t] = x\rhd[z\rhd y, t]- [z\rhd(x\rhd y), t] - [x\rhd(z\rhd t), y] + z\rhd [x\rhd t, y],
%\\
%\label{postmalcevalgebras2}
%&[x\rhd z, y\rhd t]= [\{x, y\}\rhd z, t] -  x\rhd[y\rhd z, t] +  y\rhd(x\rhd[z, t]) + [y\rhd(x\rhd t), z], \\
%\label{postmalcevalgebras3}
%&[x\rhd z, [y, t]] = [[x\rhd y, z], t] - x\rhd [[y, z], t] - [x\rhd[z, t], y] - [[x\rhd t, y], z],
%\\
%\label{postmalcevalgebras4}
%&\{y,z\}\rhd(x\rhd t)=\{\{x,y\},z\}\rhd t+y\rhd(\{x,z\}\rhd t)+x\rhd(y\rhd( z\rhd t))-z\rhd(x\rhd( y\rhd t)),
%\end{align}
%where
%$\{x, y\} = x\rhd y- y\rhd x + [x, y]$.
%If the bracket $[\cdot ,\cdot ]$ is zero, we have exactly a pre-Malcev 
%structure study in \cite{}.
 %It is worth to note that, in spite the post-Malcev product does not yield a Lie bracket by antisymmetrization,
%the bilinear product $\{\cdot,\cdot\} : A\otimes A\to A$, defined for all $x, y\in A$ by
%\begin{align}
% \{x, y\} = x \vartriangleright y - y \vartriangleright x + [x, y].
%\end{align}
%defines on $A$ another Malcev algebra structure.

\textbf{Layout of the paper}
In Section \ref{section1}, we review the definitions of commutative associative algebras, Malcev algebras, Malcev-Poisson algebras, their representations and their dual representations. In Section \ref{section2}, we review the concepts of a Manin triple for a Malcev algebra and Malcev bialgebras, with their equivalence described in terms of matched pairs of Malcev algebras.  Section \ref{section3} focuses on the bialgebra theory of Malcev-Poisson algebras, explored through the framework of matched pairs and Manin triples. These frameworks are defined using a Malcev bialgebra, a commutative infinitesimal bialgebra, and certain compatibility conditions. In Section \ref{section4},  we analyze a solution to the Malcev-Poisson Yang–Baxter equation (MPYBE) through the framework of $\mathcal{O}$-operators. Additionally, we introduce the concept of weighted relative Rota-Baxter operators on both Malcev-Poisson algebras and post-Malcev-Poisson algebras, extending the idea of post-Poisson algebras. Specifically, a  weighted relative Rota-Baxter operator on a Malcev-Poisson algebra produces a post-Malcev-Poisson algebra, while a post-Malcev-Poisson algebra naturally induces a  weighted relative Rota-Baxter operator on the corresponding Malcev-Poisson algebra.

\section{Preliminary and basics notations} \label{section1}
%In this section, we give some basic notation and recall the definition of representation Malcev algebras and dual representation \cite{Kuzmin}. We give some examples of Malcev algebras and their representations.
In this section,we review basic definitions of commutative associative algebra, Malcev algebra, Malcev-poisson algebra, their representations, and their dual representations (see \cite{Bai2010, Shestakov, Kuzmin} for more details).

Throughout the article, all vector spaces are assumed to be over an algebraically closed field $\mathbb{K}$ of characteristic $0$ and finite-dimensional. For convenience of exposition, the standard one-to-one correspondence between linear maps
$\mu : V_1\otimes\dots\otimes V_n \rightarrow W$ and multilinear maps
$\mu : V_1\times\dots\times V_n \rightarrow W$, given by
$\mu(v_1,\dots,v_n)=\mu(v_1\otimes\dots\otimes v_n)$, is assumed whenever the same notation is used for the maps.

\begin{defn}
    An algebra is a pair $(A,\circ),$ where $A$ is a linear space and $\circ: A\times A \to A$ is a bilinear map. 
    $\textbf{Associative algebras}$ are such algebras that the associator is identically zero, $as(x,y,z)=(x\circ y)\circ z-x \circ (y\circ z)=0$, that is, $ (x\circ y)\circ z=x \circ (y\circ z)$ for all $x,y,z \in A$.
   $\textbf{Commutative algebras}$ are algebras such that the commutator is identically zero, $x\circ y-x \circ y=0$, that is, $x\circ y= y\circ x$ for all $x,y \in A$.
\end{defn}
\begin{defn}
    Let $(A,\circ)$ be an associative algebra.  A representation (or a module) of $A$
is a linear space $V$ and a linear map $\mu:A\to End(V)$ such that
\begin{equation*}
    \mu(x\circ y)=\mu(x)\mu(y), \quad \forall x,y \in A.
\end{equation*} 
The map $\mu$ is also called representation of $A$ on the linear space $V$.  
\end{defn}
\begin{ex}
If $(A,\circ)$ is an associative algebra and $L: A \to End(A)$ is defined by $L(x)(y) = x \circ y$ for all $x, y \in A$, then $L(x):A\to A$ is a linear map on $A$ satisfying $L(x\circ y)(z)=(x\circ y)\circ z=x\circ (y\circ z)=L(x)L(y)(z)$ for all $x,y,z\in A$, which is equivalent to $L(x\circ y)=L(x)L(y)$ as operator equality in $End(A)$, meaning that $(A,L)$ is a representation of $(A,\circ)$. This representation is called the \textbf{regular representation} of $(A,\circ)$.
\end{ex}

 Let $V$ be a linear space. For a linear map $\theta : A \to End(V )$, define a linear map $\theta^*:A \to End(V^*)$ by
\begin{equation} \label{defeq-dualrepMA}
    \langle \theta^*(x)\xi,v\rangle=-\langle\xi,\theta(x)v\rangle,\quad \forall x\in A,\xi\in V^*,v\in V,
\end{equation}
where $\langle \cdot,\cdot \rangle$ is the ordinary pair between $V$ and $V^*$. If $(V,\mu )$ is a representation of a commutative
associative algebra $(A, \circ)$, then $(V^*,-\mu^*)$
 is also a representation of $(A, \circ)$.  In particular,  $(A^*,-L^*)$
is a representation of $(A, \circ).$

Now recall the notion of matched pairs of commutative associative algebras introduced by C. Bai  \cite{Bai2010} (See also \cite{LuiBai, NiBai}).

\begin{defn}
  If $(A_1,\circ_1)$ and $(A_2,\circ_2)$ are commutative associative algebras, and if linear maps $\mu_1:A_1\to End(A_2)$ and $\mu_2:A_2\to End(A_1)$ satisfy that $(A_2,\mu_1)$ and $(A_1,\mu_2)$ are representations of $A_1$ and $ A_2$ respectively, and for any $x_1,y_1 \in A_1$ and $x_2,y_2 \in A_2,$ 
\begin{eqnarray}
    \mu_1(x_1)(x_2\circ_2 y_2)=(\mu_1(x_1)x_2)\circ_2 y_2+\mu_1(\mu_2(x_2)x_1)y_2,\\
    \mu_2(x_2)(x_1\circ_1 y_1)=(\mu_2(x_2)x_1)\circ_1 y_1+\mu_2(\mu_1(x_1)x_2)y_1,
\end{eqnarray}
then $(A_1,A_2, \mu_1,\mu_2)$ is called a \textbf{matched pair of commutative associative algebras}. 
\end{defn}
\begin{prop}\cite{Bai2010}
 For any commutative associative algebras $(A_1,\circ_1)$ and $(A_2,\circ_2)$ and linear maps   $\mu_1:A_1\to End(A_2)$ and $\mu_2:A_2\to End(A_1)$  define a bilinear map $"\circ"$  on $A_1\oplus A_2$ by 
\begin{equation}\label{Matass}
    (x_1+x_2)\circ (y_1+y_2)=x_1 \circ_1 y_1+\mu_2(x_2)y_1+\mu_2(y_2)x_1+x_2\circ_2 y_2+\mu_1(x_1)y_2+\mu_1(y_1)x_2
    .
\end{equation}
Then $(A_1\oplus A_2, \circ)$  is a  commutative associative algebra if and only if $(A_1,A_2, \mu_1,\mu_2)$ is  a matched pair of  commutative
associative algebras. In this
case, we denote this commutative
associative algebras by  $A_1\, \bowtie A_2$. Moreover, every commutative
associative algebra which is the direct
sum of the underlying linear spaces of two subalgebras can be obtained from a matched pair of commutative
associative
algebras.
\end{prop}

\begin{defn}
  A $\textbf{Malcev algebra}$ is a non-associative algebra $A$ with an anti-symmetric bilinear multiplication (bracket) $[\cdot,\cdot]:A\times A\to A$ that satisfies the Malcev identity for $x,y,z \in A$,
\begin{equation}
\label{maltsev}
J(x,y,[x,z]) = [J(x,y,z),x],
\end{equation}
where $J(x,y,z) = [[x,y],z] + [[z,x],y] + [[y,z],x]$ is the Jacobian.

Expanding the Jacobian, since characteristic of $\mathbb{K}$ is not $2$, the Malcev identity \eqref{maltsev} is equivalent to Sagle's identity \textnormal{\cite{sagle61Malcevalgebras}},
for $x,y,z,t\in A:$
\begin{equation}\label{maltsev2}
 [[x, z], [y, t]] = [[[x, y], z], t] + [[[y, z], t], x] + [[[z, t], x], y] + [[[t, x], y], z].   
\end{equation}
 \end{defn}
 
%Let $A$ be the simple Malcev algebra over the field of complex numbers $\mathbb{C}$ {\rm \cite[Example 3]{goncharov}}. In this case $A$ has a basis  $\{e_1,e_2,e_3,e_4,e_5, e_6, e_7\}$ with the following table of multiplication:
%\begin{center}
%\begin{tabular}{c|c|c|c|c|c|c|c}
%$[\cdot, \cdot]$ & $e_1$ & $e_2$ & $e_3$ & $e_4$ & $e_5$ & $e_6$ & $e_7$ \\
%\hline
%$e_1$ & $0$ & $2e_2$ & $-2e_3$ & $2e_4$ & $-2e_5$ & $2e_6$ & $-2e_7$\\
%\hline
%$e_2$ & $-2e_2$ & $0$ & $e_1$ & $2e_7$ & $0$ & $-2e_5$ & $0$  \\
%\hline
%$e_3$ & $2e_3$ & $-e_1$ & $0$ & $0$ & $-2e_6$ & $0$ & $2e_4$ \\
% \hline
%$e_4$ & $-2e_4$ & $-2e_7$ & $0$ & $0$ & $e_1$ & $2e_3$ & $0$ \\
%  \hline
%$e_5$ & $2e_5$ & $0$ & $2e_6$ & $-e_1$ & $0$ & $0$ & $-2e_2$\\
% \hline
%$e_6$ & $-2e_6$ & $2e_5$ & $0$ & $-2e_3$ & $0$ & $0$ & $e_1$\\
% \hline
%$e_7$ & $2e_7$ & $0$ & $-2e_4$ & $0$ & $2e_2$ & $-e_1$ & $0$
%\end{tabular}.
%\end{center}

\begin{ex} \label{ex-Malcevalgebras}
\begin{enumerate}[label=\upshape{\arabic*.}, ref=\upshape{\arabic*}, labelindent=5pt, leftmargin=*] 
\item \label{ex1-Malcevalgebras}
Lie algebras are Malcev algebras since the multiplication is anti-symmetric and Jacobian $J(x,y,z)=0$, for all $x,y,z\in A$ 
\item \label{ex2-Malcevalgebras} 
Alternative algebras (that is algebras in which the associator is alternative mapping) are Malcev admissible, that is 
any alternative algebra $(A,*)$ is a Malcev algebra with respect to the commutator product $[x,y]=x*y-y*x$. (see \textnormal{\cite{sagle61Malcevalgebras,myung,ElduqueMyung-MutAltAlgs1994}})
\end{enumerate}
\end{ex}

\begin{ex}[\!\!\cite{hegazi}]\label{ex1}
 Let $A$ be a $6$-dimensional linear space with basis $\{e_1,e_2,e_3,e_4,e_5,e_6\}.$  Define  the non-zero braket operation on $A$ by:
    \begin{equation*}
        [e_1,e_2]=e_3, \quad [e_1,e_3]=e_6,\quad [e_2,e_4]=e_6,\quad [e_3,e_4]=e_5.
    \end{equation*}
    Then  $(A, [\cdot,\cdot])$  is a Malcev algebra.
\end{ex}
 
\begin{defn} \label{df:repmalcevalgebra}
A $\textbf{representation}$ (or $\textbf{a module}$) of a Malcev algebra $(A, [\cdot, \cdot])$ on
a linear space $V$ is a linear map $\varrho: A\rightarrow End(V)$ such that, for all
  $x,y\in A$,
\begin{equation}\label{representation}
\varrho([[x, y], z]) = \varrho(z)\varrho(y)\varrho(x) - \varrho(y)\varrho(x)\varrho(z) + \varrho(x)\varrho([y,z]) + \varrho([x, z])\varrho(y).
\end{equation}
We denote this  representation  by $(V, \varrho)$.
\end{defn}
\begin{ex} \label{ex-RepresntationsMalcevalgebras}
    Let $(A,[\cdot,\cdot])$ a Malcev algebra and $\mathrm{ad}:A \to End(A)$ be a linear map defined by
$\mathrm{ad}(x)(y)=[x,y],\ 
\forall x,y\in A.$
Then $(A,\mathrm{ad})$ is a representation of the Malcev algebra,  which is called the \textbf{adjoint representation}. In fact, the identity \eqref{maltsev2} can be written as follows, for $x,y,z\in A,$
\begin{align*}
    \mathrm{ad}([[x, y], z]) (t)&= \mathrm{ad}(z)\mathrm{ad}(y)\mathrm{ad}(x)(t) - \mathrm{ad}(y)\mathrm{ad}(x)\mathrm{ad}(z)(t) \\&
    + \mathrm{ad}(x)\mathrm{ad}([y,z])(t) + \mathrm{ad}([x, z])\mathrm{ad}(y)(t).
    \end{align*}
\end{ex}

Representations of a Malcev algebra can be characterized in terms of semi-direct product Malcev algebras as follows.

\begin{prop}
  Let $(A,[\cdot,\cdot])$ be a Malcev algebra, $V$ be a linear space and $\varrho:A\to End(V)$ be a linear map. Then $(V, \varrho)$ is a representation of a Malcev algebra $(A,[\cdot,\cdot])$ if and only if $A\oplus V$ is a Malcev algebra
(the \textbf{semi-direct product}) with the multiplication on $A\oplus V$ by
\begin{equation}
    [x+u,y+v]=[x,y]+\varrho(x)v-\varrho(y)u, \quad \forall x,y\in A,u,v\in V.
\end{equation}
\end{prop}

\begin{ex}
    Let  $(A,[\cdot,\cdot])$ be a Malcev algebra and $(V, \varrho)$ a representation of $(A ,[\cdot,\cdot])$. Then  $(A\oplus V,\varrho_{A\oplus V})$ is a representation of $(A,[\cdot,\cdot])$ where $\varrho_{A\oplus V}:A \to End(A\oplus V)$ is defined by 
    \begin{equation*}
      \varrho_{A\oplus V}(x)(y +v)=[x, y] +\varrho(x)v.  
    \end{equation*}
     Indeed,  
     \begin{align*}
       \varrho_{A\oplus V}([[x, y], z])(t+v) &-\varrho_{A\oplus V}(z)\varrho_{A\oplus V}(y)\varrho_{A\oplus V}(x)(t+v) + \varrho_{A\oplus V}(y)\varrho_{A\oplus V}(x)\varrho_{A\oplus V}(z)(t+v)\\& - \varrho_{A\oplus V}(x)\varrho_{A\oplus V}([y,z])(t+v) - \varrho_{A\oplus V}([x, z])\varrho_{A\oplus V}(y)(t+v)\\&=[[[x,y],z],t]+\rho([[x,y],z])v-[z,[y,[x,t]]]-\rho(z)\rho(y)\rho(x)v\\&+[y,[x,[z,t]]]+\rho(y)\rho(x)\rho(z)v-[x,[[y,z],t]]-\rho(x)\rho([y,z])v\\&-[[x,z],[y,t]]-\rho([x,z])\rho(y)v =0.   
     \end{align*}
\end{ex}

To relate matched pairs of Malcev algebras to Malcev bialgebras and Manin triples for Malcev algebras in the next section, we need the notions of the coadjoint representation, which is the dual
representation of the adjoint representation. 
 %Let $V$ be a linear space. For a linear map $\varrho : A \to End(V )$, define a linear map $\varrho^*:A \to End(V^*)$ by
%\begin{equation} \label{defeq-dualrepMA}
  %  \langle \varrho^*(x)\xi,v\rangle=-\langle\xi,\varrho(x)v\rangle,\quad \forall x\in A,\xi\in V^*,v\in V,
%\end{equation}
%where $\langle \cdot,\cdot \rangle$ is the ordinary pair between $V$ and $V^*$.
%With the above notation, we get the following statement. 
\begin{prop} \label{prop-dualrepMA}
    Let $(V,\varrho)$ be a representation of a Malcev algebra $(A,[\cdot,\cdot])$, then $(V^*,\varrho^*)$ is  a representation of $(A,[\cdot,\cdot]).$
\end{prop}
Consider the case when $V=A$  and define the linear map $\mathrm{ad}_A^*:A\to End(A^*)$ by 
 \begin{equation}
     \langle\mathrm{ad}_A^*(x)\xi,y\rangle=-\langle\xi,\mathrm{ad}(x)(y)\rangle, \quad\forall x,y \in A, \xi\in A^*.
 \end{equation}

 \begin{cor} \label{cor-coadjointrepMA}
 Let $(A,[\cdot,\cdot])$ a Malcev algebra and $(A,\mathrm{ad})$ be the adjoint representation. Then $(A^*,\mathrm{ad}_A^*)$ is a representation of $(A,[\cdot,\cdot])$ on $A^*$ which is called the coadjoint representation.
\end{cor}

\begin{ex}
The Example \ref{ex-RepresntationsMalcevalgebras} 
of the representations of Malcev algebras leads to the follwing dual representations in the sense of 
\eqref{defeq-dualrepMA} and Proposition \ref{prop-dualrepMA}, and coadjoint representations of the Malcev algebras in the sense of the Corollary \ref{cor-coadjointrepMA}.       
\end{ex}

Now we consider the case of Malcev-Poisson algebras.
\begin{defn}
  \textbf{A Malcev-Poisson algebra} $(A,[\cdot,\cdot], \circ)$ is a
linear space equipped with two bilinear operations $[\cdot,\cdot],\circ:A\times A \to A$ such that $(A, [\cdot,\cdot])$ is a Malcev
algebra and $(A,\circ)$ is a commutative associative algebra and the  Leibniz rule is satisfied, that is, 
\begin{equation} \label{eq:malpoisalg}
    [x,y\circ z]=[x, y] \circ z+ y\circ [x,z], \quad\quad \forall\  x,y,z\in A.
\end{equation}   
\end{defn}
\begin{ex}
    Let $A$ be a $6$-dimensional linear space with basis $\{e_1,e_2,e_3,e_4,e_5,e_6\}.$ For the Malcev algebra $(A, [\cdot,\cdot])$ in Example \ref{ex1}, define the non-zero multiplication  on $A$ by 
    \begin{equation*}
        e_1\circ e_1=e_5, \quad e_2\circ e_2=e_6,\quad 
        e_1\circ e_4=e_5, \quad e_4\circ e_4=e_6,
    \end{equation*}
Then  $(A,\circ, [\cdot,\cdot])$  is a Malcev-Poisson algebra.
\end{ex}
\begin{defn}
    Let $(A, [\cdot,\cdot], \circ)$ be a Malcev-Poisson algebra, $V$ be a linear space and $\varrho, \mu : A \to End (V)$ be
two linear maps. Then $(V, \varrho, \mu)$ is called a \textbf{representation} (or module) of $A$ if $(V, \varrho)$
is a representation of $(A, [\cdot,\cdot])$ and $(V, \mu)$ is a representation of $(A, \circ)$ and they are compatible in the
sense that (for any $x, y \in A$):
\begin{equation}\label{eq:repmalcpoisalg1}
   \varrho(x\circ y)=\mu(y)\varrho(x)+\mu(x)\varrho(y),
\end{equation}
\begin{equation}\label{eq:repmalcpoisalg2}
    \mu([x,y])=\varrho(x)\mu(y)-\mu(y)\varrho(x).
\end{equation}
\end{defn}
\begin{thm}
    Under the notation, the direct sum $A \oplus V$ is turned into a Malcev-Poisson algebra by defining the operations by (we still denote
the operations by $[\cdot,\cdot]$ and $\circ$):
\begin{equation}
    [x_1+v_1,x_2+v_2]=[x_1,x_2]+\varrho(x_1)v_2-\varrho(x_2)v_1,
\end{equation}
\begin{equation}
    (x_1+v_1)\circ(x_2+v_2)=x_1\circ x_2+ \mu(x_1)v_2+ \mu(x_2)v_1,
\end{equation}
for all $x_1, x_2 \in A,v_1,v_2 \in V$. We denote it by $A\ltimes  V.$
\end{thm}

\begin{ex}
  If $(V, \varrho, \mu)$ is a representation of a Malcev-Poisson algebra $(A, [\cdot,\cdot], \circ),$  then $(V^*, \varrho^*, \mu^*)$ is a representation of a Malcev-Poisson algebra $(A, [\cdot,\cdot], \circ)$, too. Therefore, both $(A,\mathrm{ad}, L)$ and $(A^*,\mathrm{ad}_A^*, -L^*)$ are representations of a Malcev-Poisson algebra $(A, [\cdot,\cdot], \circ).$  
\end{ex}

\section{Matched pairs, Manin triples of Malcev algebras and  Malcev bialgebras}\label{section2}
In this section, we study  the notions of  Manin triple of a Malcev algebra and
Malcev bialgebras. The equivalence between them is interpreted in terms of matched pairs of Malcev algebras.
Let us first recall the notion of matched pairs of Malcev algebras \cite{zhang}.
Let $\left( {A_1, [\cdot,\cdot]_1} \right)$ and $\left({A_2, [\cdot,\cdot]_2}\right)$ be Malcev algebras, and let $\varrho_1: A_1\to End(A_2)$ and $\varrho_2: A_2\to End(A_1)$ be linear maps.
On the direct sum $A_1\oplus A_2$ of the underlying linear space, define a linear map $[\cdot,\cdot]: (A_1\oplus A_2) \times (A_1\oplus A_2) \rightarrow A_1\oplus A_2$ by 
\small{\begin{equation}\label{MatMal}
    [x_1+x_2,y_1+y_2]=[x_1,y_1]_1+\varrho_2(x_2)y_1-\varrho_2(y_2)x_1+[x_2,y_2]_2+\varrho_1(x_1)y_2-\varrho_1(y_1)x_2.
\end{equation}}
\begin{thm}\label{matched}
   Under the above notation, $(A_1\oplus A_2,[\cdot,\cdot])$ is a Malcev algebra if and only if $(A_2,\varrho_1)$ and $(A_1,\varrho_2)$  are representations of $\left( {A_1, [\cdot,\cdot]_1} \right)$ and $\left({A_2, [\cdot,\cdot]_2}\right)$ respectively  and for all $x_1,y_1 \in A_1\;\; \text{and} \;\;x_2,y_2 \in A_2,$  the following compatibility conditions are satisfied:
\begin{align}
    \nonumber
    &[[\varrho_2(x_2)x_1,y_1]_1,z_1]_1-[\varrho_2(x_2)[y_1,z_1]_1,x_1]_1-[\varrho_2(\varrho_1(x_1)x_2)y_1,z_1]_1\\
    \nonumber
    &\quad -[[\varrho_2(x_2)z_1,x_1]_1,y_1]_1+\varrho_2(\varrho_1(y_1))(\varrho_1(x_1)x_2)z_1+[[x_1,z_1]_1,\varrho_2(x_2)y_1]_1\\
    \nonumber
    &\quad -\varrho_2(x_2)[[x_1,y_1],z_1]_1+\varrho_2(\varrho_1(y_1)x_2)[x_1,z_1]_1+[\varrho_2(\varrho_1(z_1)x_2)x_1,y_1]_1\\
    &\quad +\varrho_2(\varrho_1([y_1,z_1])x_2)x_1-\varrho_2(\varrho_1(x_1))(\varrho_1(z_1)x_2)y_1=0,
    \label{matched1}
\\*[0.2cm]
    \nonumber
    &[[\varrho_1(x_1)x_2,y_2]_2,z_2]_2-[\varrho_1(x_1)[y_2,z_2]_2,x_2]_2-[\varrho_1(\varrho_2(x_2)x_1)y_2,z_2]_2\\
    \nonumber
    &\quad -[[\varrho_1(x_1)z_2,x_2]_2,y_2]_2+\varrho_1(\varrho_2(y_2))(\varrho_2(x_2)x_1)z_2+[[x_2,z_2]_2,\varrho_1(x_1)y_2]_2\\
    \nonumber
    &\quad -\varrho_1(x_1)[[x_2,y_2]_2,z_2]_2+\varrho_1(\varrho_2(y_2)x_1)[x_2,z_2]_2+[\varrho_1(\varrho_2(z_2)x_1)x_2,y_2]_2\\
    &\quad +\varrho_1(\varrho_2([y_2,z_2])x_1)x_2-\varrho_1(\varrho_2(x_2))(\varrho_2(z_2)x_1)y_2=0,
    \label{matched2}
\\*[0.2cm]
\nonumber
&\varrho_2(\varrho_1(x_1)[x_2,y_2]_2)y_1-\varrho_2(\varrho_1(x_1)x_2)\varrho_2(y_2)y_1+\varrho_2(x_2)[\varrho_2(y_2)x_1,y_1]_1\\
\nonumber
&\quad -[\varrho_2(y_2)\varrho_2(x_2)y_1,x_1]_1-[\varrho_2([x_2,y_2]_2)x_1,y_1]_1+\varrho_2(\varrho_1(y_1)y_2)(\varrho_2(x_2)x_1)\\
\nonumber
&\quad -\varrho_2([\varrho_1(y_1)x_2,y_2]_2)x_1-\varrho_2(y_2)(\varrho_2(x_2)[x_1,y_1]_1)+[\varrho_2(x_2)x_1,\varrho_2(y_2)y_1]_1\\
&\quad -\varrho_2(x_2)(\varrho_2(\varrho_1(x_1)y_2)y_1)+\varrho_2(\varrho_1(\varrho_2(x_2)y_1)y_2)x_1=0,\label{matched3}
\\*[0.2cm]
\nonumber
    &\varrho_1(\varrho_2(x_2)[x_1,y_1]_1)y_2-\varrho_1(\varrho_2(x_2)x_1)\varrho_1(y_1)y_2+\varrho_1(x_1)[\varrho_1(y_1)x_2,y_2]_2\\
\nonumber
    &\quad -[\varrho_1(y_1)\varrho_1(x_1)y_2,x_2]_2-[\varrho_1([x_1,y_1]_1)x_2,y_2]_2+\varrho_1(\varrho_2(y_2)y_1)(\varrho_1(x_1)x_2)\\
\nonumber
    &\quad -\varrho_1([\varrho_2(y_2)x_1,y_1]_1)x_2-\varrho_1(y_1)(\varrho_1(x_1)[x_2,y_2]_2)+[\varrho_1(x_1)x_2,\varrho_1(y_1)y_2]_2\\
    &\quad -\varrho_1(x_1)(\varrho_1(\varrho_2(x_2)y_1)y_2)+\varrho_1(\varrho_2(\varrho_1(x_1)y_2)y_1)x_2=0,\label{matched4}
\\*[0.2cm]
\nonumber
    &\varrho_2(x_2)[\varrho_2(y_2)y_1,x_1]_1-\varrho_2([\varrho_1(x_1)y_2,x_2]_2)y_1-[\varrho_2(x_2)(\varrho_2(y_2)x_1),y_1]_1\\
    \nonumber
    &\quad+ \varrho_2(y_2)[\varrho_2(x_2)x_1,y_1]_1-[\varrho_2(y_2)(\varrho_2(x_2)y_1),x_1]_1-\varrho_2([\varrho_1(y_1)x_2,y_2]_2)x_1\\
    \nonumber
    &\quad- \varrho_2(y_2)(\varrho_2(\varrho_1(x_1)x_2)y_1)-\varrho_2(x_2)(\varrho_2(\varrho_1(y_1)y_2)x_1)+\varrho_2([x_2,y_2]_2)[x_1,y_1]_1\\ &\quad+ \varrho_2(\varrho_1(\varrho_2(x_2)y_1)y_2)x_1+\varrho_2(\varrho_1(\varrho_2(y_2)x_1)x_2)y_1=0,\label{matched5}
\\*[0.2cm]
\nonumber
    &\varrho_1(x_1)[\varrho_1(y_1)y_2,x_2]_2-\varrho_1([\varrho_2(x_2)y_1,x_1]_1)y_2-[\varrho_1(x_1)(\varrho_1(y_1)x_2),y_2]_2\\
    \nonumber
    &\quad+\varrho_1(y_1)[\varrho_1(x_1)x_2,y_2]_2-[\varrho_1(y_1)(\varrho_1(x_1)y_2),x_2]_2-\varrho_1([\varrho_2(y_2)x_1,y_1]_1)x_2\\
    \nonumber
    &\quad-\varrho_1(y_1)(\varrho_1(\varrho_2(x_2)x_1)y_2)-\varrho_1(x_1)(\varrho_1(\varrho_2(y_2)y_1)x_2)+\varrho_1([x_1,y_1]_1)[x_2,y_2]_2\\ &\quad+\varrho_1(\varrho_2(\varrho_1(x_1)y_2)y_1)x_2+\varrho_1(\varrho_2(\varrho_1(y_1)x_2)x_1)y_2=0.\label{matched6}
\end{align}

\end{thm}
\begin{defn}
A \textbf{matched pair} of Malcev  algebras is a quadruple $(A_1,A_2, \varrho_1,\varrho_2)$  consisting
of Malcev algebras $(A_1,[\cdot,\cdot]_1)$ and $(A_2,[\cdot,\cdot]_2)$, together with representations $\varrho_1: A_1\to End(A_2)$ and $\varrho_2: A_2\to End(A_1)$  respectively, such that the compatibility conditions
\eqref{matched1}-\eqref{matched6} are satisfied.
\end{defn}
\begin{re}
We denote the Malcev algebra defined by    \eqref{MatMal} by  $A_1\, \bowtie A_2$. It
is straightforward to show that every Malcev algebra which is a direct sum of the underlying linear spaces
of two  Malcev subalgebras can be obtained from a
matched pair of Malcev algebras as above.
\end{re}
\begin{defn}\label{inva}
A bilinear form $\mathcal{B}$ on a Malcev algebra $(A,[\cdot,\cdot])$ is called invariant if it satisfies, for any $x,y,z\in A$:
\begin{equation}
    \mathcal{B}([x,y],z)=\mathcal{B}(x,[y,z]).
\end{equation}
\end{defn}
\begin{defn}
     A Manin triple of Malcev algebras $(A, A^ {+} , A^{-})$ is a triple of Malcev algebras $A,\quad
A^{+}, \quad\text{and}\quad A^{-} $ together with a non-degenerate symmetric invariant bilinear form $\mathcal{B}(\cdot,\cdot)$ on $A$ such that the following conditions satisfied:
\begin{enumerate}[label=\upshape{\arabic*.}, ref=\upshape{\arabic*}, labelindent=5pt, leftmargin=*]
    \item $A^{+}$  and $A^{-}$  are Malcev subalgebras of $A$,
    \item $A = A^{+} \oplus A^{-} $ as linear spaces,
    \item $A^{+}$ and $A^{-}$ are isotropic with respect to $\mathcal{B}(\cdot,\cdot)$.
\end{enumerate}
\end{defn}
\begin{defn}
Let $(A,[\cdot,\cdot])$ be a Malcev algebra. Suppose that there is a Malcev algebra structure $(A^*,[\cdot,\cdot]^*)$ on the dual space $A^*$ of $A$ and there is a Malcev algebra structure on the direct sum $A\oplus A^*$ of the underlying linear spaces $A$ and $A^*$ such that $(A,[\cdot,\cdot])$ and $(A^*,[\cdot,\cdot]^*)$ are subalgebras and the natural non-degerenate symmetric bilinear form on $A\oplus A^*$ given by
\begin{equation}\label{standard bilinear form}
    \mathcal{B}_d(x+\xi,y+\eta):=\langle x,\eta\rangle +\langle \xi,y\rangle,\quad  \forall x,y\in A,\;\xi ,\eta\in A^*,
\end{equation}
is invariant, then $(A\oplus A^*,A,A^*)$ is called a standard Manin triple of Malcev algebra associated to standard bilinear form $\mathcal{B}_d$.
\end{defn}
Obviously,    a standard Manin triple of Malcev algebras is a Manin triple of Malcev
algebras.

\begin{prop} \label{matman}
Let $(A,[\cdot,\cdot])$ be a Malcev algebra. Suppose that there is a Malcev algebra structure $(A^*,[\cdot,\cdot]^*)$ on $A^*$. Then there exists a Malcev algebra structure on the linear space $A\oplus A^*$ such that $(A\oplus A^*,A,A^*)$ is a standard Manin triple of Malcev  algebras with respect to $\mathcal{B}_d$ defined by \eqref{standard bilinear form} if and only if $(A,A^*,\mathrm{ad}_A^*,\mathrm{ad}_{A^*}^*)$  is a matched pair of Malcev algebras. Here $\mathrm{ad}_{A^*}^*$ is the coadjoint representation of the Malcev algebra $(A^*,[\cdot,\cdot]^*)$.
\end{prop}

For a Malcev algebras $(A^*, [\cdot,\cdot]^*),$  let $\delta: A\to A\otimes A$ be the dual map of $[\cdot,\cdot]^*:\otimes^2 A^* \to A^*$ and $\delta^*:\otimes^2 A^* \to A^*$, i.e.,
\begin{equation}
    \langle \delta(x), \xi \otimes \eta\rangle=\langle x,[\xi,\eta]^*\rangle=\langle x,\delta^*(\xi\otimes\eta)\rangle.
\end{equation}

\begin{defn}
A \textbf{Malcev coalgebra} is a pair
$(A,\delta)$, such that $A$ is a linear space and
$\delta:A\rightarrow A\otimes A$ is a linear map satisfying
\begin{align}\label{eq:antisymmetric}
\tau\delta &= -\delta,
\\ 
\label{eq:MalCo}
(\delta\otimes \delta)\delta &=(\mathrm{id}+\sigma+\sigma^{2}+\sigma^{3})(\delta\otimes \mathrm{id}\otimes\mathrm{id})(\delta \otimes\mathrm{id})\delta,
\end{align}
where $\sigma:A\otimes A\otimes A\otimes A\rightarrow A\otimes A\otimes A\otimes A$ is the linear map
defined for all
$x,y,z,t\in A$ as $$\sigma(x\otimes y\otimes z\otimes t)=y\otimes z\otimes t \otimes x.$$ 
\end{defn}

\begin{prop}
Let $(A,[\cdot,\cdot])$ be a Malcev algebra. Then $(A,\delta)$  is a Malcev coalgebra if and only if  $(A^*,\delta^*)$  is a Malcev algebra.
\end{prop}

\begin{defn}\label{de:MalBia}
A \textbf{Malcev bialgebra} is a triple $(A,[\cdot,\cdot],\delta)$, such that:
\begin{enumerate}[label=\upshape{\arabic*)}, ref=\upshape{\arabic*}, labelindent=5pt, leftmargin=*]
\item $(A,[\cdot,\cdot])$ is a Malcev algebra,
\item $(A,\delta)$ is a Malcev coalgebra,
\item $\delta$ satisfies the followings conditions, for all $x,y,z\in A$,
\begin{align}\label{eq:MalBia}
&\delta([[x,y],z])+\delta([y,z])(\mathrm{id}\otimes \mathrm{ad}(x))+(\mathrm{ad}(y)\otimes \mathrm{id})\delta([x,z])\nonumber\\
&\quad = \big(\mathrm{ad}(\mathrm{ad}(z)(y))\otimes \mathrm{id}+\mathrm{ad}(z)\otimes \mathrm{ad}(y)+\mathrm{id}\otimes\mathrm{ad}(z)\mathrm{ad}(y)\big)\delta(x)\nonumber\\
&\quad -\big(\mathrm{ad}(z)\mathrm{ad}(x)\otimes \mathrm{id}+\mathrm{ad}(x)\otimes\mathrm{ad}(z)-\mathrm{id}\otimes\mathrm{ad}(\mathrm{ad}(x)(z))\big)\delta(y)\nonumber\\
&\quad +\big(\mathrm{ad}(x)\mathrm{ad}(y)\otimes\mathrm{id}-\mathrm{id}\otimes\mathrm{ad}(y)\mathrm{ad}(x)\big)\delta(z),
\\*[0,2cm] 
\label{eq:MalBia1}
  &(\mathrm{id}\otimes \delta)\delta([x,y])=(\mathrm{id}\otimes\mathrm{id}\otimes\mathrm{ad}(x))((\mathrm{id}\otimes \delta)\delta(y))+(\delta\otimes\mathrm{id})((\mathrm{ad}(x)\otimes\mathrm{id})\delta(y)\nonumber\\
  &\quad + (\mathrm{id}\otimes[\cdot,\cdot] \otimes\mathrm{id})((\mathrm{id}\otimes\tau \otimes\mathrm{id})(\delta(y)\otimes\delta(x))-((\delta\otimes\mathrm{id})\delta(x))(\mathrm{ad}(y)\otimes\mathrm{id}\otimes\mathrm{id})\nonumber\\
  &\quad -((\delta\otimes\mathrm{id})\delta(x))(\mathrm{id}\otimes\mathrm{id}\otimes\mathrm{ad}(y))+(\mathrm{id}\otimes\mathrm{id}\otimes[\cdot,\cdot] )((\mathrm{id}\otimes\tau \otimes\mathrm{id})(\delta(x)\otimes\delta(y))\nonumber\\
  &\quad -(\mathrm{id}\otimes\tau)((\mathrm{id}\otimes\mathrm{id}\otimes\mathrm{ad}(x))(\delta\otimes\mathrm{id})\delta(y))+(\mathrm{id}\otimes\tau)((\delta\otimes\mathrm{id})(\delta(x)(\mathrm{ad}(y)\otimes\mathrm{id}))\nonumber\\
  &\quad -(\mathrm{id}\otimes\mathrm{ad}(y)\otimes\mathrm{id})((\mathrm{id}\otimes \delta)\delta(x))+(\mathrm{id}\otimes\tau)(((\delta\otimes\mathrm{id})\delta(y))(\mathrm{ad}(x)\otimes\mathrm{id}\otimes\mathrm{id})).
\end{align}
\end{enumerate}
\end{defn}
%\begin{defn}
%\begin{enumerate}[label=\upshape{\arabic*)}, ref=\upshape{\arabic*}, labelindent=5pt, leftmargin=*]
 %\item A \textbf{$\mathcal{ML^*}$ bialgebra} is a triple $(A,[\cdot,\cdot],\delta)$, such that
%\begin{enumerate}[label=\upshape{\roman*)}, ref=\upshape{\roman*}, labelindent=5pt, leftmargin=*]
%\item $(A,[\cdot,\cdot])$ is a Malcev algebra,
%\item $(A,\delta)$ is a Lie coalgebra,
%\item $\delta$ satisfies \eqref{eq:MalBia} and \eqref{eq:MalBia1}  
%\end{enumerate}

%\item  A \textbf{$\mathcal{LM^*}$ bialgebra} is a triple $(A,[\cdot,\cdot],\delta)$, such that
%\begin{enumerate}[label=\upshape{\roman*)}, ref=\upshape{\roman*}, labelindent=5pt, leftmargin=*]
%\item $(A,[\cdot,\cdot])$ is a Lie algebra,
%\item $(A,\delta)$ is a Malcev coalgebra,
%\item $\delta$ satisfies \eqref{eq:MalBia} and \eqref{eq:MalBia1}  
%\end{enumerate}
%\end{enumerate}
%\end{defn}
\begin{ex}
    \begin{enumerate}
        \item Lie bialgebras are Malcev bialgebras  (See \textnormal{\cite{drinfeld}} for more details about Lie bialgebras). 
        \item  \textbf{$\mathcal{LM^*}$} bialgebras are Malcev bialgebras since $(A,[\cdot,\cdot])$ is a Lie algebra and $\delta$ satisfies \eqref{eq:MalBia} and \eqref{eq:MalBia1}.
        \item \textbf{$\mathcal{ML^*}$}  bialgebras are Malcev bialgebras since and $(A,\delta)$ is a Lie coalgebra and $\delta$ satisfies \eqref{eq:MalBia} and \eqref{eq:MalBia1}. 
            {\bf\large$$\text{Lie Bialg}\hookrightarrow\textbf{$\mathcal{LM^*}$}\ \text{Bialg}\hookrightarrow \text{Malcev Bialg} \hookleftarrow\textbf{$\mathcal{ML^*}$}\ \text{Bialg} \hookleftarrow \text{Lie Bialg.}$$}

    \end{enumerate}
\end{ex}
\begin{thm}[\hspace{-0.1pt}\cite{versh}]\label{MalcevBial}
Let $(A,[\cdot,\cdot])$ be a Malcev algebra. Suppose that there is a Malcev  algebra
structure on $A^*$ denoted by $[\cdot,\cdot]^*$ which is defined as a linear map $\delta : A\rightarrow A\otimes A$. Then
the following conditions are equivalent:
\begin{enumerate}[label=\upshape{\arabic*)}, ref=\upshape{\arabic*}, labelindent=5pt, leftmargin=*]
    \item \label{MalcevBial-1} $(A,A^*;\mathrm{ad}_A^*,\mathrm{ad}_{A^*}^*)$  is a matched pair of Malcev algebras.
     \item \label{MalcevBial-2} $(A,[\cdot,\cdot],\delta)$ is a Malcev bialgebra.
     \item \label{MalcevBial-3} $(A\oplus A^*,A,A^*)$ is a standard  Manin triple of Malcev algebras with respect to $\mathcal{B}_d$ defined by    \eqref{standard bilinear form}.
\end{enumerate}
\end{thm}
\begin{proof}
    First we prove that \ref{MalcevBial-1} is equivalent to \ref{MalcevBial-2}. Obviously,   \eqref{eq:MalBia} is exactly   \eqref{matched1}, and furthermore \eqref{matched4}  and \eqref{matched6} and   \eqref{eq:MalBia1} is exactly  \eqref{matched2},\eqref{matched3} and \eqref{matched5}  in the case $\varrho_1=\mathrm{ad}_A^*, \varrho_2=\mathrm{ad}_{A^*}^*.$
    In fact, by   \eqref{matched1}, for any $x,y,z \in A, \quad\xi,\eta,\gamma \in A^*$, 
   \begin{align*}
   &\langle [[\mathrm{ad}_{A^*}^*(\xi)x,y],z],\eta\rangle - \langle[\mathrm{ad}_{A^*}^*(\xi)[y,z],x],\eta\rangle-\langle[\mathrm{ad}_{A^*}^*(\mathrm{ad}_A^*(x)\xi)y,z],\eta\rangle\\
   &\quad -
   \langle[[\mathrm{ad}_{A^*}^*(\xi)z,x],y],\eta\rangle+\langle\mathrm{ad}_{A^*}^*(\mathrm{ad}_A^*(y))(\mathrm{ad}_A^*(x)\xi)z,\eta\rangle+\langle[[x,z],\mathrm{ad}_{A^*}^*(\xi)y],\eta\rangle\\
   &\quad-\langle \mathrm{ad}_{A^*}^*(\xi)[[x,y],z],\eta\rangle+\langle\mathrm{ad}_{A^*}^*(\mathrm{ad}_A^*(y)\xi)[x,z],\eta\rangle+\langle[\mathrm{ad}_{A^*}^*(\mathrm{ad}_A^*(z)\xi)x,y],\eta\rangle\\
   &\quad+\langle\mathrm{ad}_{A^*}^*(\mathrm{ad}_A^*([y,z])\xi)x,\eta\rangle- \langle\mathrm{ad}_{A^*}^*(\mathrm{ad}_A^*(x))(\mathrm{ad}_A^*(z)\xi)y,\eta\rangle\\
   &=\langle\mathrm{ad}(z)\mathrm{ad}(y)(\mathrm{ad}_{A^*}^*(\xi)x),\eta\rangle+\langle\mathrm{ad}(\mathrm{ad}_{A^*}^*(\xi)[y,z]),\eta\rangle+\langle\mathrm{ad}(z)(\mathrm{ad}_{A^*}^*(\mathrm{ad}_A^*(x)\xi)y),\eta\rangle\\
   &\quad-\langle\mathrm{ad}(y)\mathrm{ad}(x)(\mathrm{ad}_{A^*}^*(\xi)z),\eta\rangle-\langle z,[\mathrm{ad}_A^*(y)\mathrm{ad}_A^*(x)\xi,\eta]^*\rangle+\langle\mathrm{ad}([x,z])\mathrm{ad}_{A^*}^*(\xi)y,\eta\rangle\\
   &\quad+\langle[[x,y],z],[\xi,\eta]^*\rangle-\langle[x,z],[\mathrm{ad}_A^*(y)\xi,\eta]^*\rangle-\langle \mathrm{ad}(y)\mathrm{ad}_{A^*}^*(\mathrm{ad}_A^*(z)\xi)x),\eta\rangle \\
   &\quad- \langle x,[\mathrm{ad}_A^*([y,z])\xi,\eta]^*\rangle+\langle y,[\mathrm{ad}_A^*(x)\mathrm{ad}_A^*(z)\xi,\eta]^*\rangle\\
   &=-\langle \delta(x),\xi\otimes \mathrm{ad}_A^*(y)\mathrm{ad}_A^*(z)\eta\rangle+\langle\delta([y,z]),\xi\otimes \mathrm{ad}_A^*(x)\eta\rangle+\langle\delta(y),\mathrm{ad}_A^*(x)\xi\otimes\mathrm{ad}_A^*(z)\eta\rangle\\
   &\quad+\langle\delta(z),\xi\otimes \mathrm{ad}_A^*(x)\mathrm{ad}_A^*(y)\eta\rangle-\langle\delta(z), \mathrm{ad}_A^*(y)\mathrm{ad}_A^*(x)\xi\otimes\eta\rangle+\langle\delta(y),\xi\otimes\mathrm{ad}_A^*([x,z])\eta\rangle\\
   &\quad+\langle\delta([[x,y],z]),\xi\otimes\eta\rangle+\langle\delta([x,z]),\mathrm{ad}_A^*(y)\xi\otimes\eta\rangle-\langle\delta(x),\mathrm{ad}_A^*(z)\xi\otimes\mathrm{ad}_A^*(y)\eta\rangle\\
   &\quad+\langle\delta(x),\mathrm{ad}_A^*([y,z])\xi\otimes\eta\rangle+\langle\delta(y),\mathrm{ad}_A^*(x)\mathrm{ad}_A^*(z)\xi\otimes\eta\rangle\\
   &=-\langle(\mathrm{id}\otimes\mathrm{ad}(z)\mathrm{ad}(y))\delta(x),\xi\otimes\eta\rangle-\langle\delta([y,z])(\mathrm{id}\otimes \mathrm{ad}(x)),\xi\otimes\eta\rangle+\langle(\mathrm{ad}(x)\otimes\mathrm{ad}(z))\delta(y),\xi\otimes\eta\rangle\\
   &\quad+\langle(\mathrm{id}\otimes\mathrm{ad}(y)\mathrm{ad}(x))\delta(z),\xi\otimes\eta\rangle-\langle(\mathrm{ad}(x)\mathrm{ad}(y))\delta(z),\xi\otimes\eta\rangle-\langle(\mathrm{id}\otimes\mathrm{ad}(\mathrm{ad}(x)z)\delta(y),\xi\otimes\eta\rangle\\
   &\quad+\langle\delta([[x,y],z]),\xi\otimes\eta\rangle-\langle(\mathrm{ad}(y)\otimes\mathrm{id})\delta([x,z]),\xi\otimes\eta\rangle-\langle(\mathrm{ad}(z)\otimes\mathrm{ad}(y))\delta(x),\xi\otimes\eta\rangle\\
   &\quad-\langle(\mathrm{ad}(\mathrm{ad}(y))z\otimes\mathrm{id})\delta(x),\xi\otimes\eta\rangle+\langle(\mathrm{ad}(z)\mathrm{ad}(x)\otimes\mathrm{id})\delta(y),\xi\otimes\eta\rangle=0.
   \end{align*}
  Also, by   \eqref{matched3}, 
   \begin{align*}
       &\langle \mathrm{ad}_{A^*}^*(\mathrm{ad}_A^*(x)[\xi,\eta]^*,\gamma\rangle-\langle \mathrm{ad}_{A^*}^*(\mathrm{ad}_A^*(x)\xi)\mathrm{ad}_{A^*}^*(\eta)y,\gamma\rangle+\langle \mathrm{ad}_{A^*}^*(\xi)[\mathrm{ad}_{A^*}^*(\eta)x,y],\gamma\rangle\\
       &\quad -\langle [\mathrm{ad}_{A^*}^*(\eta)\mathrm{ad}_{A^*}^*(\xi)y,x],\gamma\rangle-\langle [\mathrm{ad}_{A^*}^*([\xi,\eta]^*)x,y],\gamma\rangle+\langle \mathrm{ad}_{A^*}^*(\mathrm{ad}_A^*(y)\eta)\mathrm{ad}_{A^*}^*(\xi)x,\gamma\rangle\\
       &\quad-\langle \mathrm{ad}_{A^*}^*([\mathrm{ad}_A^*(y)\xi,\eta]^*)x,\gamma\rangle-\langle \mathrm{ad}_{A^*}^*(\eta)(\mathrm{ad}_{A^*}^*(\xi)[x,y],\gamma\rangle+\langle [\mathrm{ad}_{A^*}^*(\xi)x,(\mathrm{ad}_{A^*}^*(\eta)y],\gamma\rangle\\
       &\quad-\langle \mathfrak(\xi)\mathrm{ad}_{A^*}^*(\mathrm{ad}_A^*(x)\eta)y,\gamma\rangle+\langle \mathrm{ad}_{A^*}^*(\mathrm{ad}_A^*(\mathrm{ad}_{A^*}^*(\xi)y)\eta)x,\gamma\rangle\\
       &=-\langle \delta(y),\mathrm{ad}_A^*(x)[\xi,\eta]^*\otimes \gamma\rangle+\langle \delta(y),[\mathrm{ad}_A^*(x)\xi,\gamma]^*\otimes \eta\rangle-\langle \delta(x),\mathrm{ad}_A^*(y)[\xi,\gamma]^*\otimes \eta\rangle\\
       &\quad-\langle \delta(y),\xi \otimes[\eta,\mathrm{ad}_A^*(x)\gamma]^*\rangle+\langle \delta(x),[\xi,\eta]^*\otimes\mathrm{ad}_A^*(y)\gamma\rangle+\langle \delta(x),\xi\otimes[\mathrm{ad}_A^*(y)\eta,\gamma]^*\rangle\\
       &\quad+\langle \delta(x),[\mathrm{ad}_A^*(y)\xi,\eta]^*\otimes \gamma\rangle-\langle \delta([x,y]),\xi \otimes[\eta,\gamma]^*\rangle-\langle \delta(x),\xi \otimes\mathrm{ad}_A^*(\mathrm{ad}_{A^*}^*(\eta)y)\gamma\rangle\\
       &\quad+\langle \delta(y),[\xi,\gamma]^*\otimes\mathrm{ad}_A^*(x)\eta\rangle-\langle \delta(x),\mathrm{ad}_A^*(\mathrm{ad}_{A^*}^*(\xi)y)\eta\otimes\gamma\rangle\\
       &=\langle (\delta\otimes\mathrm{id})(\mathrm{ad}(x)\otimes\mathrm{id})\delta(y)+(\mathrm{id}\otimes \tau) ((\delta\otimes\mathrm{id})\delta(y))(\mathrm{ad}(x)\otimes \mathrm{id} \otimes\mathrm{id})\\
       &\quad+(\mathrm{id}\otimes \tau) ((\delta\otimes\mathrm{id})\delta(x))(\mathrm{ad}(y)\otimes\mathrm{id})+(\mathrm{id} \otimes \mathrm{id}\otimes\mathrm{ad}(x))((\mathrm{id}\otimes\delta)\delta(y))\\
       &\quad-((\delta\otimes\mathrm{id})\delta(x))(\mathrm{id}\otimes\mathrm{id}\otimes\mathrm{ad}(y))- (\mathrm{id}\otimes\mathrm{ad}(y)\otimes\mathrm{id})((\mathrm{id}\otimes\delta)\delta(x))\\
       &\quad-(\delta\otimes\mathrm{id})\delta(x)(\mathrm{ad}(y)\otimes\mathrm{id}\otimes\mathrm{id})- (\mathrm{id}\otimes\delta)\delta([x,y])\\
       &\quad+(\mathrm{id}\otimes\mathrm{id}\otimes[\cdot,\cdot])(\mathrm{id}\otimes\tau\otimes\mathrm{id})(\delta(x)\otimes\delta(y))-(\mathrm{id}\otimes\tau)(\mathrm{id}\otimes\mathrm{id} \otimes \mathrm{ad}(x))((\delta\otimes\mathrm{id})\delta(y))\\
       &\quad+(\mathrm{id}\otimes[\cdot,\cdot]\otimes\mathrm{id})(\mathrm{id}\otimes\tau\otimes\mathrm{id})(\delta(y)\otimes\delta(x)),\xi\otimes\eta\otimes\gamma\rangle =0.
   \end{align*}
By the same way, we can prove that \eqref{matched4} and \eqref{matched6} in Theorem \ref{matched} are equivalent to \eqref{eq:MalBia} in Malcev bialgebra, and \eqref{matched2} and \eqref{matched5}  are equivalent to    \eqref{eq:MalBia1}. 
\end{proof}
\section{Bialgebras, Matched pairs and  Manin triples of a Malcev-Poisson algebra}\label{section3}
In this section, we explore the bialgebra theory of Malcev-Poisson algebras through the frameworks of matched pairs and Manin triples, which are defined using a Malcev bialgebra, a commutative infinitesimal bialgebra, and a specified compatibility conditions.
\subsection{Matched pairs of Malcev-Poisson algebras}

\begin{defn}
 Let $(A_1, [\cdot,\cdot]_1, \circ_1)$ and $(A_2, [\cdot,\cdot]_2, \circ_2)$ be two Malcev-Poisson algebras. Let
$\varrho_1, \mu_1: A_1\to End(A_2)$ and $\varrho_2, \mu_2: A_2\to End(A_1)$ be four linear maps such that $(A_1, A_2, \varrho_1,\varrho_2)$ is a matched pair of  Malcev algebra and $(A_1, A_2, \mu_1,\mu_2)$ is a matched pair of commutative associative algebra.If in addition, $(A_2, \varrho_1 ,\mu_1 )$ and $(A_1,\varrho_2 ,\mu_2 )$ are representations of the Malcev-Poisson algebras
$(A_1, [\cdot,\cdot]_1, \circ_1)$ and $(A_2, [\cdot,\cdot]_2, \circ_2)$ respectively, and $\varrho_1 ,\varrho_2 ,\mu_1 ,\mu_2$ are compatible in the following
sense:
\begin{align}
&  \varrho_2(x_2)(x_1\circ_1 y_1)=(\varrho_2(x_2)x_1)\circ_1 y_1+x_1\circ_1 (\varrho_2(x_2)y_1)\nonumber\\
&\hspace{3cm} -\mu_2(\varrho_1(x_1)x_2)y_1-\mu_2(\varrho_1(y_1)x_2)x_1,\label{matpoi1}
\\*[0,2cm]
& \varrho_1(x_1)(x_2\circ_2 y_2)=(\varrho_1(x_1)x_2)\circ_2 y_2+x_2\circ_2 (\rho_1(x_1)y_2)\nonumber\\
&\hspace{3cm} -\mu_1(\varrho_2(x_2)x_1)y_2-\mu_1(\varrho_2(y_2)x_1)x_2,\label{matpoi2}
\\*[0,2cm]
&  [x_1,\mu_2(x_2)y_1]_1-\varrho_2(\mu_1(y_1)x_2)x_1=\mu_2(\varrho_1(x_1)x_2)y_1-(\varrho_2(x_2)x_1)\circ_1 y_1\nonumber\\
&\hspace{6cm} +\mu_2(x_2)[x_1,y_1]_1,\label{matpoi3}
\\*[0,2cm]
& [x_2,\mu_1(x_1)y_2]_2-\varrho_1(\mu_2(y_2)x_1)x_2=\mu_1(\varrho_2(x_2)x_1)y_2-(\varrho_1(x_1)x_2)\circ_2 y_2\nonumber\\
&\hspace{6cm} +\mu_1(x_1)[x_2,y_2]_2\label{matpoi4},
\end{align}
for any $x_1, y_1 \in A_1$ and $x_2,y_2 \in A_2$. Such a structure is called a \textbf{matched pair of Malcev-Poisson
algebras} $A_1$ and $A_2$. We denote it by     $(A_1, A_2, \varrho_1, \mu_1,\varrho_2,\mu_2)$.
\end{defn}

\begin{prop}\label{matchedpair}
   Let $(A_1, [\cdot,\cdot]_1, \circ_1)$ and $(A_2, [\cdot,\cdot]_2, \circ_2)$ be two Malcev-Poisson algebras. For linear maps $\varrho_1, \mu_1: A_1\to End(A_2)$ and $\varrho_2, \mu_2: A_2\to End(A_1)$, define two operation $"\circ"$ and $[\cdot,\cdot]$ on $A_1\oplus A_2$ by    \eqref{Matass} and   \eqref{MatMal}  respectively. Then $(A_1\oplus A_2 , [\cdot,\cdot], \circ )$ is a Malcev-Poisson algebra if and only if $(A_1, A_2, \varrho_1, \mu_1,\varrho_2,\mu_2)$ is a matched pair of Malcev-Poisson
algebras. In this case, we denote this Malcev-Poisson algebra by  $A_1\, \bowtie A_2$. Moreover, every Malcev-Poisson algebra which is a direct sum of the underlying spaces of two subalgebras
can be obtained from a matched pair of Malcev-Poisson algebras.
\end{prop}

\begin{proof}
It is known in \cite{Bai2010} that $(A_1\oplus A_2 ,  \circ )$ is a commutative associative algebra if and only if $(A_1, A_2, \mu_1,\mu_2)$
is a matched pair of commutative associative algebras. Similarly, by Theorem \ref{matched}, $(A_1\oplus A_2 , [\cdot,\cdot])$ is a Malcev algebra if and
only if $(A_1, A_2, \varrho_1, \varrho_2)$ is a matched pair of Malcev algebras.
Let $x_1,y_1,z_1 \in A_1, \quad  x_2, y_2,z_2 \in A_2,$  we consider the Leibniz rule \eqref{eq:malpoisalg} on $A_1 \oplus A_2$
\begin{align*}
  & [x_1+x_2,(y_1+y_2)\circ (z_1+z_2)]-(z_1+z_2)\circ[x_1+x_2, y_1+y_2]  - (y_1+y_2)\circ [x_1+x_2,z_1+z_2]\\
  =&[x_1+x_2,y_1 \circ_1 z_1+\mu_2(y_2)z_1+\mu_2(z_2)y_1+y_2\circ_2 z_2+\mu_1(y_1)z_2+\mu_1(z_1)x_2]\\
  &- (z_1+z_2)\circ\big([x_1,y_1]_1+\varrho_2(x_2)y_1-\varrho_2(y_2)x_1+[x_2,y_2]_2+\varrho_1(x_1)y_2-\varrho_1(y_1)x_2\big) 
  \\
  &-(y_1+y_2)\circ \big([x_1,z_1]_1+\varrho_2(x_2)z_1-\varrho_2(z_2)x_1+[x_2,z_2]_2+\varrho_1(x_1)z_2-\varrho_1(z_1)x_2\big)
  \\=&[x_1,y_1\circ_1 z_1]_1+[x_1,\mu_2(y_2)z_1]_1+[x_1,\mu(z_2)y_1]_1+\varrho_2(x_2)(y_1\circ_1 z_1)+\varrho_2(x_2)\mu_2(y_2)z_1\\&+\varrho_2(x_2)\mu_2(z_2)y_1-\varrho_2(y_2\circ_2 z_2)x_1-\varrho_2(\mu_1(y_1)z_2)x_1-\varrho_2(\mu_1(z_1)y_2)x_1+ [x_2,y_2 \circ_2 z_2]_2\\&+[x_2,\mu_1(y_1)z_2]_2+[x_2,\mu_1(z_1)y_2]_2+\varrho_1(x_1)(y_2 \circ_2 z_2)+\varrho_1(x_1)\mu_1(y_1)z_2+\varrho_1(x_1)\mu_1(z_1)y_2\\&-\varrho_1(y_1 \circ_1 z_1)x_2-\varrho_1(\mu_2(y_2)z_1)x_2-\varrho_1(\mu_2(z_2)y_1)x_2-z_1\circ_1[x_1,y_1]_1-z_1\circ_1 \varrho_2(x_2)y_1\\&+z_1\circ_1\varrho_2(y_2)x_1-\mu_2(z_2)[x_1,y_1]_1-\mu_2(z_2)\varrho_2(x_2)y_1+\mu_2(z_2)\varrho_2(y_2)x_1-\mu_2([x_2,y_2]_2)z_1\\&-\mu_2(\varrho_1(x_1)y_2)z_1+\mu_2(\varrho_1(y_1)x_2)z_1-z_2\circ_2 [x_2,y_2]_2-z_2\circ_2\varrho_1(x_1)y_2+z_2\circ_2\varrho_1(y_1)x_2\\&-\mu_1(z_1)[x_2,y_2]_2-\mu_1(z_1)\varrho_1(x_1)y_2+\mu_1(z_1)\varrho_1(y_1)x_2-\mu_1([x_1,y_1]_1)z_2-\mu_1(\varrho_2(x_2)y_1)z_2\\&+\mu_1(\varrho_2(y_2)x_1)z_2-y_1\circ_1[x_1,z_1]_1-y_1\circ_1 \varrho_2(x_2)z_1+y_1\circ_1\varrho_2(z_2)x_1-\mu_2(y_2)[x_1,z_1]_1\\&-\mu_2(y_2)\varrho_2(x_2)z_1+\mu_2(y_2)\varrho_2(z_2)x_1-\mu_2([x_2,z_2]_2)y_1-\mu_2(\varrho_1(x_1)z_2)y_1+\mu_2(\varrho_1(z_1)x_2)y_1\\&-y_2\circ_2 [x_2,z_2]_2-y_2\circ_2\varrho_1(x_1)z_2+y_2\circ_2\varrho_1(z_1)x_2-\mu_1(y_1)[x_2,z_2]_2-\mu_1(y_1)\varrho_1(x_1)z_2\\&+\mu_1(y_1)\varrho_1(z_1)x_2-\mu_1([x_1,z_1]_1)y_2-\mu_1(\varrho_2(x_2)z_1)y_2+\mu_1(\varrho_2(z_2)x_1)y_2.
\end{align*}
Then   \eqref{eq:malpoisalg}  hold for $(A_1,[\cdot,\cdot]_1,\circ_1)$  and $(A_2,[\cdot,\cdot]_2, \circ_2)$ as Malcev-Poisson algebras, $(A_2, \varrho_1 ,\mu_1 )$ and $(A_1,\varrho_2 ,\mu_2 )$ as representations of the Malcev-Poisson algebras
$(A_1, [\cdot,\cdot]_1, \circ_1)$ and $(A_2, [\cdot,\cdot]_2, \circ_2)$ respectively and the following equations holds:
\begin{align*}
&  \varrho_2(x_2)(y_1\circ_1 z_1)=(\varrho_2(x_2)y_1)\circ_1 z_1+y_1\circ_1 (\varrho_2(x_2)z_1)\\
&\hspace{3cm} -\mu_2(\varrho_1(y_1)x_2)z_1-\mu_2(\varrho_1(z_1)x_2)y_1,
\\*[0,2cm]
& \varrho_1(x_1)(y_2\circ_2 z_2)=(\varrho_1(x_1)y_2)\circ_2 z_2+y_2\circ_2 (\rho_1(x_1)z_2)\\
&\hspace{3cm} -\mu_1(\varrho_2(y_2)x_1)z_2-\mu_1(\varrho_2(z_2)x_1)y_2,
\\*[0,2cm]
&  [x_1,\mu_2(y_2)z_1]_1-\varrho_2(\mu_1(z_1)y_2)x_1=\mu_2(\varrho_1(x_1)y_2)z_1-(\varrho_2(y_2)x_1)\circ_1 z_1\\
&\hspace{6cm} +\mu_2(y_2)[x_1,z_1]_1,
\\*[0,2cm]
& [x_2,\mu_1(y_1)z_2]_2-\varrho_1(\mu_2(z_2)y_1)x_2=\mu_1(\varrho_2(x_2)y_1)z_2-(\varrho_1(y_1)x_2)\circ_2 z_2\\
&\hspace{6cm} +\mu_1(y_1)[x_2,z_2]_2,
\end{align*}
Hence the conclusion holds.
%The verification is a routine calculation and thus omitted.
\end{proof}

\subsection{Manin triples of Malcev-Poisson algebras.}
\begin{defn}
     A Manin triple of Malcev-Poisson algebras $(A, A^{+} , A^{-})$ is a triple of Malcev-Poisson algebras $A$, $A^{+}$ and $A^{-} $ together with a nondegenerate symmetric bilinear form $\mathcal{B}(\cdot,\cdot)$ on $A$ which is invariant
in the sense that
\begin{equation}\label{bilinearform1}
     \omega(x\circ y,z)=\omega(x,y\circ z),
\end{equation}
\begin{equation}\label{bilinearform2}
    \omega([x,y],z)=\omega(x,[y,z]),
\end{equation}
satisfying the following conditions:
\begin{enumerate}[label=\upshape{\arabic*)}, ref=\upshape{\arabic*}, labelindent=5pt, leftmargin=*]
    \item $A^{+}$  and $A^{-}$  are Malcev-Poisson subalgebras of $A$,
    \item $A = A^{+} \oplus A^{-} $ as linear spaces,
    \item $A^{+}$ and $A^{-}$ are isotropic with respect to $\omega(\cdot,\cdot)$.
\end{enumerate}
\end{defn}
A homomorphism between two Manin triples of  Malcev-Poisson algebras  $(A_1,A^+_1,A^-_1)$ and $(A_2,A^+_2,A^-_2)$ associated to two nondegenerate symmetric invariant bilinear forms $\omega_1$ and $\omega_2$ respectively, is a homomorphism of Malcev-Poisson algebras $\phi:A_1 \rightarrow A_2$ such that 
\begin{align*}
      \phi(A^+_1)\subseteq A^+_2,\quad \phi(A^-_1)\subseteq A^-_2,\quad \omega_1(x,y)=\omega_2(\phi(x),\phi(y)),\ \forall x,y \in A.
\end{align*}
If in addition, $\phi$ is an isomorphism of linear spaces, then the two Manin triples are called isomorphic.

Obviously, a Manin triple of Malcev-Poisson algebras is just a triple of Malcev-Poisson algebras such that they
are both a Manin triple of Malcev algebras and a commutative associative version of Manin triple
with the same nondegenerate symmetric bilinear form (and share the same isotropic subalgebras).
Moreover, it is easy to see that $A^+$  and $A^-$ are Lagrangian subalgebras of $A$.

In particular, there is a special (standard) Manin triple of Malcev-Poisson algebras as follows. Let
$(A, [\cdot,\cdot], \circ)$ be a Malcev-Poisson algebra. If there is a Malcev-Poisson algebra structure on the direct sum of the
underlying linear space of $A$ and its dual space $A^*$ such that $(A \oplus A^*, A, A^*)$ is a Manin triple of
Malcev-Poisson algebras with the invariant symmetric bilinear form on $A \oplus A^*$ given by:
\begin{equation}\label{stanMalPoi}
\omega_d(x+\xi,y+\eta)=\langle x,\eta\rangle + \langle \xi,y\rangle , \;\; \forall  x,y \in A,\,\,\xi,\eta \in A^*.
\end{equation}
Then $(A \oplus A^*, A , A^*)$  is called a standard Manin triple of Malcev-Poisson algebras.

Obviously, a standard Manin triple of Malcev-Poisson algebras is a Manin triple of Malcev-Poisson algebras.
Conversely, it is easy to show that every Manin triple of Malcev-Poisson algebras is isomorphic to a standard
one. Furthermore, it is straightforward to get the following structure theorem.

\begin{thm}\label{matpair}
   Let $(A, [\cdot,\cdot]_1, \circ_1)$ and $(A^*, [\cdot,\cdot]_2, \circ_2)$ be two Malcev-Poisson algebras. Then $(A \oplus A^*, A , A^*)$  is a standard Manin triple of Malcev-Poisson algebras if and only if\\  $(A, A^*, \mathrm{ad}_A^*, -L^*,\mathrm{ad}_{A^*}^*,\mathcal{-L}_{A^*}^*)$ is a matched pair of Malcev-Poisson algebras.
\end{thm}
\begin{proof}
It is known in \cite{Bai2010}, there is a commutative associative algebra structure on $A\oplus A^*$
such
that both $(A, \circ_1)$ and $(A^*
, \circ_2 )$ are commutative associative subalgebras and $\omega$ satisfies   \eqref{bilinearform1}
if and only if $(A, A^*
, -L^*,\mathcal{-L}_{A^*}^*)$ is a matched pair of commutative associative algebras. Similarly by Proposition \ref{matman}, $(A\oplus A^*,A,A^*)$ is a standard Manin triple of Malcev  algebras with respect to $\mathcal{B}_d$ defined by \eqref{standard bilinear form} if and only if $(A,A^*,\mathrm{ad}_A^*,\mathrm{ad}_{A^*}^*)$ of Malcev algebras. Hence if $(A \oplus A^*, A , A^*)$  is a standard Manin triple of Malcev-Poisson algebras, then by Proposition \ref{matchedpair} with 
\begin{equation*}
    A_1=A,\quad A_2=A^*, \quad\mu_1=-L^*,\quad \rho_1=\mathrm{ad}_A^*, \quad\mu_2=\mathcal{-L}_{A^*}^*,\quad \rho_2=\mathrm{ad}_{A^*}^*,
\end{equation*}
we get $(A, A^*, \mathrm{ad}_A^*, -L^*,\mathrm{ad}_{A^*}^*,\mathcal{-L}_{A^*}^*)$ is a matched pair of Malcev-Poisson algebras. Conversely, if $(A, A^*, \mathrm{ad}_A^*, -L^*,\mathrm{ad}_{A^*}^*,\mathcal{-L}_{A^*}^*)$ is a matched pair of Malcev-Poisson algebras, then by Propostion \ref{matchedpair} again, there is a Malcev-Poisson algebra $A\oplus A^*$  obtained from the
matched pair with both $A$ and $A^*$ as Malcev-Poisson subalgebras.  Moreover, the bilinear
form $\omega_d$ is invariant on $A\oplus A^*$. Hence $(A \oplus A^*, A , A^*)$  is a standard Manin triple of Malcev-Poisson algebras.
\end{proof}

\subsection{Malcev-Poisson bialgebras.}
\begin{defn}
A \textbf{cocommutative coassociative coalgebra} is a pair
$(A,\Delta)$, such that $A$ is a linear space and
$\Delta:A\rightarrow A\otimes A$ is a linear map satisfying
\begin{align}\label{eq:symmetric}
\tau\Delta &= \Delta,
\\
\label{AssoCo}
(\mathrm{id}\otimes \Delta)\Delta &= (\Delta\otimes\mathrm{id})\Delta,
\end{align}
where $\tau: A\otimes A\rightarrow A\otimes A$ is the exchanging operator defined as $\tau(x\otimes
y)=y\otimes x,$  for all $x,y\in A$.
\end{defn}

\begin{defn}
A \textbf{commutative and cocommutative infinitesimal bialgebra}
is a triple $(A,\circ,\Delta)$ such that: 
\begin{enumerate}[label=\upshape{\arabic*)}, ref=\upshape{\arabic*}, labelindent=5pt, leftmargin=*]
\item $(A,\circ)$ is a commutative associative algebra,
\item $(A,\Delta)$ is a cocommutative coassociative coalgebra,
\item $\Delta$ satisfies the following condition:
\begin{equation}\label{AssoBia}
\Delta(x\circ y)=( L(x)\otimes \mathrm{id})\Delta(y)+(\mathrm{id}\otimes L(y))\Delta(x),\;\;\forall x,y\in A.
\end{equation}
\end{enumerate}
\end{defn}

Now, we give the definitions of a Malcev-Poisson coalgebra and a Malcev-Poisson bialgebra.

\begin{defn}
Let $A$ be a linear space  and $\Delta,\delta:A\rightarrow A\otimes A.$ Then $(A,\Delta,\delta)$ is called a \textbf{Malcev-Poisson coalgebra} if $(A,\Delta)$ is a cocommutative coassociative coalgebra, $(A,\delta)$ is a Malcev coalgebra, and the following conditions are satisfied
\begin{equation}\label{eq:Co1}
(\mathrm{id}\otimes\Delta)\delta(x)-(\delta\otimes
\mathrm{id})\Delta(x)-(\tau\otimes
\mathrm{id})(\mathrm{id}\otimes\delta)\Delta(x)=0, \;\;\forall x\in A.
\end{equation}
\end{defn}
\begin{prop}
Under the finite-dimensional assumption, $(A,\Delta,\delta)$ is a Malcev-Poisson coalgebra if and only  if  $(A^*,\Delta^*,\delta^*)$ is a Malcev-Poisson algebra.
\end{prop}
\begin{proof}
Obviously, $(A,\Delta)$ is a cocommutative coassociative coalgebra if and only if $(A^*,\Delta^*)$ is a
commutative associative algebra and $(A,\delta)$ is a Malcev coalgebra if and only if $(A^*, \delta^*)$ is a Malcev  algebra. For all $ \xi ,\eta 
 \in A^*,$ set $\Delta^*(\xi \otimes \eta)=\xi \circ \eta,\quad \delta^*(\xi\otimes \eta)=[\xi,\eta]$. Then,
\begin{align*}
    &\langle(\mathrm{id}\otimes\Delta)\delta(x), \xi \otimes\eta \otimes \gamma\rangle=\langle x ,\delta^*(\mathrm{id}\otimes\Delta^*)(\xi \otimes\eta \otimes \gamma)\rangle=\langle x, [\xi,\eta\circ \gamma]\rangle,\\&
\langle(\delta\otimes
\mathrm{id})\Delta(x),\xi \otimes\eta \otimes \gamma\rangle=\langle x ,\Delta^*(\delta^* \otimes\mathrm{id})(\xi \otimes\eta \otimes \gamma)\rangle=\langle x, [\xi,\eta]\circ \gamma\rangle,\\&
\langle(\tau\otimes
\mathrm{id})(\mathrm{id}\otimes\delta)\Delta(x), \xi \otimes\eta \otimes \gamma\rangle=\langle x ,\Delta^*(\mathrm{id}\otimes\delta^*)(\tau\otimes
\mathrm{id})(\xi \otimes\eta \otimes \gamma)\rangle=\langle x, \eta\circ[\xi, \gamma]\rangle.
\end{align*} 
for all $x \in A \quad \text{and} \quad\xi, \eta, \gamma \in A^*.$ Thus  \eqref {eq:malpoisalg}  hold for $(A^*, \delta^*, \Delta^*)$  as a Malcev-Poisson algebra if and only if \eqref{eq:Co1} hold.
\end{proof}
\begin{defn}\label{de:MPBA}
A \textbf{Malcev-Poisson bialgebra} is a collection $(A,[\cdot,\cdot],\circ,\Delta,\delta)$ satisfying the following conditions: 
\begin{enumerate}[label=\upshape{\arabic*)}, ref=\upshape{\arabic*}, labelindent=5pt, leftmargin=*]
\item $(A,[\cdot,\cdot],\circ)$ is a Malcev-Poisson algebra,
\item $(A,\Delta,\delta)$ is a Malcev-Poisson coalgebra,
\item $\Delta$ satisfies   \eqref{AssoBia} and hence $(A,\circ,\Delta)$ is a commutative and commutative infinitesimal bialgebra,
\item  $\delta$ satisfies   \eqref{eq:MalBia},  \eqref{eq:MalBia1}  and hence  $(A,[\cdot,\cdot],\delta)$ is a Malcev bialgebra;
\item The following equations hold, for all $x,y\in A$,
\begin{eqnarray}
&&
\delta(x\circ y)+(\mathrm{ad}(y)\otimes\mathrm{id})\Delta(x)-(\mathrm{id}\otimes L(x))\delta(y)\nonumber\\
&&+(\mathrm{ad}(x)\otimes\mathrm{id})\Delta(y)-(\mathrm{id}\otimes L(y))\delta(x)=0,
\label{eq:Bi4}\\
&&\Delta([x,y])-(L(y)\otimes \mathrm{id})\delta(x)-(\mathrm{id}\otimes\mathrm{ad}(x))\Delta(y), \nonumber\\
&&+(\mathrm{id}\otimes L(y))\delta(x)-(\mathrm{ad}(x)\otimes \mathrm{id})\Delta(y)=0.
\label{eq:Bi5}
\end{eqnarray}
\end{enumerate}
\end{defn}
\begin{ex}
\begin{enumerate}
    \item  Any Poisson bialgebra is a Malcev-Poisson bialgebra.
    \item A Malcev bialgebra $(A,[\cdot,\cdot],\delta)$ is a Malcev-Poisson bialgebra $(A,[\cdot,\cdot],\circ,\Delta,\delta)$ in which the commutative associative algebra $(A,\circ)$ is trivial, that is $$x\circ y=0,\ \forall x,y \in A.$$
\end{enumerate}   
\end{ex}

\begin{thm}
   Let $(A,[\cdot,\cdot],\circ)$ be a Malcev-Poisson algebra. Suppose that there is a Malcev-Poisson  algebra
structure on $A^*$ denoted by $\Delta^*\quad\text{and}\quad \delta^*$ which are defined as a linear maps $\Delta : A\rightarrow A\otimes A\quad \text{and}\quad \delta : A\rightarrow A\otimes A  $. Then
the following conditions are equivalent:
\begin{enumerate}[label=\upshape{\arabic*)}, ref=\upshape{\arabic*}, labelindent=5pt, leftmargin=*]
     \item $(A,[\cdot,\cdot],\circ, \Delta,\delta)$ is a Malcev-Poisson bialgebra,
    \item $(A, A^*, \mathrm{ad}_A^*, -L^*,\mathrm{ad}_{A^*}^*,\mathcal{-L}_{A^*}^*)$  is a matched pair of Malcev-Poisson algebras,    
     \item $(A\oplus A^*,A,A^*)$ is a standard  Manin triple of Malcev-Poisson algebras with respect to $\omega_d$ defined by    \eqref{stanMalPoi}.
\end{enumerate} 
\end{thm}
\begin{proof}
According to Theorem \ref{matpair}, the second condition is equivalent to the third condition. Now, we only need to prove that the first condition holds if and only if the second condition holds. In fact, it is known that $(A,\circ, \Delta)$ is a commutative and cocommutative infinitesimal bialgebra if and only if  $(A, A^*, -L^*,\mathcal{-L}_{A^*}^*)$  is a matched pair of commutative associative algebra and according to Theorem \ref{MalcevBial}, $(A,[\cdot,\cdot], \delta)$ is a Malcev bialgebra if and only if $(A, A^*, \mathrm{ad}^*,\mathrm{ad}_{A^*}^*)$  is a matched pair of Malcev algebras. Now, it remains to prove  \eqref{matpoi1} - \eqref{matpoi4} are equivalent to  \eqref{eq:Bi4} and \eqref{eq:Bi5}.
 By    \eqref{matpoi1}, for any $x,y\in A,$ and $\xi, \eta \in A^*,$  
  \begin{align*}
      &\langle L(y)\mathrm{ad}_{A^*}^*(\xi)x+L(x)\mathrm{ad}_{A^*}^*(\xi)y+ \mathcal{L}_{A^*}^*(\mathrm{ad}_{A^*}^*(x)\xi)y+\mathcal{L}_{A^*}^*(\mathrm{ad}_{A^*}^*(y)\xi)x- \mathrm{ad}_{A^*}^*(\xi)(x\circ y),\eta\rangle\\=&\langle \delta(x),\xi\otimes L^*(y)\eta\rangle+\langle \delta(y),\xi\otimes L^*(x)\eta\rangle-\langle \Delta(y),\mathrm{ad}_A^*(x)\xi\otimes\eta\rangle\\&+\langle \Delta(x),\mathrm{ad}_A^*(y)\xi\otimes\eta\rangle-\langle \delta(x\circ y),\xi\otimes\eta\rangle\\=&\langle -(\mathrm{id}\otimes L(y))\delta(x)-(\mathrm{id}\otimes L(x))\delta(y)+(\mathrm{ad}(x)\otimes \mathrm{id})\Delta(y)+(\mathrm{ad}(y)\otimes \mathrm{id})\Delta(x)\\&+\delta(x\circ y),\xi\otimes\eta\rangle=0.
  \end{align*}
  Again, by   \eqref{matpoi3},  
  \begin{align*}
      &\langle -\mathcal{L}_{A^*}^*(\mathrm{ad}_A^*(x)\xi)y-L(y)\mathrm{ad}_{A^*}^*(\xi)x-\mathcal{L}_{A^*}^*(\xi)[x,y]+\mathrm{ad}(x)\mathcal{L}_{A^*}^*(\xi)y-\mathrm{ad}_{A^*}^*(L^*(y)\xi)x,\eta\rangle\\=&\langle\Delta(y),\mathrm{ad}_A^*(x)\xi\otimes\eta\rangle-\langle\delta(x),\xi\otimes L^*(y)\eta\rangle+\langle\Delta([x,y]),\xi\otimes\eta\rangle\\&+\langle\Delta(y),\xi\otimes\mathrm{ad}_A^*(x)\eta\rangle+\langle\delta(x),L^*(y)\xi\otimes\eta\rangle\\=&\langle-(\mathrm{ad}(x)\otimes\mathrm{id})\Delta(y)+(\mathrm{id}\otimes L(y))\delta(x)+\Delta([x,y])-(\mathrm{id}\otimes\mathrm{ad}(x))\Delta(y)\\&-(L(y)\otimes\mathrm{id})\delta(x), \xi\otimes\eta\rangle=0.
  \end{align*}
  The same way, we can prove that the identity \eqref{matpoi4} is equivalent to the identity \eqref{eq:Bi4}, and the identity \eqref{matpoi2} is  equivalent to the identity \eqref{eq:Bi5}.
\end{proof}

\section{Operator forms of the MPYBE and  relative Rota-Baxter operators of Malcev-Poisson algebras}\label{section4}
In this section, we interpret a solution of the Malcev-Poisson Yang–Baxter equation (MPYBE) in terms of relative Rota-Baxter operators. Also, we introduce notions of  $\lambda$-weighted relative Rota-Baxter operators on a Malcev-Poisson algebra and post-Malcev-Poisson algebra as a generalization of post-Poisson algebra. On one hand, $\lambda$-weighted relative Rota-Baxter operator on a Malcev-Poisson gives a post-Malcev-Poisson algebra, and on the other hand, a post-Malcev-Poisson algebra naturally gives rise to a $\lambda$-weighted relative Rota-Baxter operator on the sub-adjacent Malcev-Poisson algebra.
\subsection{Operator forms of the MPYBE}
Let $(A, \circ, [\cdot,\cdot])$ be a Malcev-Poisson  algebra  and $r=\sum\limits_i x_i\otimes y_i\in A^{\otimes 2}$. Set
\begin{equation*}
r_{12}=\sum_ix_i\otimes y_i\otimes 1,\quad r_{13}=\sum_{i}x_i\otimes
1\otimes y_i,\quad r_{23}=\sum_i1\otimes x_i\otimes y_i,
\end{equation*}
where $1$ is a unit element if $(A,\circ, [\cdot,\cdot])$ is unital or a symbol playing a similar role of
the unit for the nonunital cases. In the sense that
$$r_{12}\circ r_{13}=(\sum_ix_i\otimes y_i\otimes 1)\circ( \sum_{j}x_j\otimes
1\otimes y_j)=\sum_{i,j}x_i\circ x_j\otimes y_i\otimes y_j,$$
\begin{equation*}
r_{13}\circ r_{23}=\sum_{i,j}x_i\otimes
x_j\otimes y_i\circ y_j,\;
r_{23}\circ r_{12} =\sum_{i,j}x_j\otimes
x_i\circ y_j\otimes y_i\,,
\end{equation*}
and
\begin{equation*}
[r_{12},r_{13}]=\sum_{i,j}[x_i,x_j]\otimes y_i\otimes
y_j,\; [r_{13}, r_{23}]=\sum_{i,j}x_i\otimes x_j\otimes
[y_i, y_j],\;[r_{12}, r_{23}]=\sum_{i,j}x_i\otimes
[y_i,  x_j]\otimes y_j.\end{equation*}
\begin{defn}
Let $(A, \circ, [\cdot,\cdot])$ be a Malcev-Poisson  algebra  and $r\in A\otimes A$. $r$ satisfies Malcev-Poisson Yang-Baxter equation
(MPYBE) if
\begin{equation}
C_{A}(r)=r_{12}\circ r_{13}-r_{23}\circ r_{12}+r_{13}\circ r_{23}=0,
\label{rslt1}
\end{equation}
\begin{equation}
M_{A}(r)=[r_{12},r_{13}]+[r_{12},r_{23}]+[r_{13},r_{23}]= 0,
\label{rslt2}
\end{equation}
\end{defn}

Equations \eqref{rslt1} and \eqref{rslt2} imply that
$r$ satisfies both the associative Yang-Baxter equation (AYBE) (See \textnormal{\cite{Aguiar1,Aguiar2,Aguiar3}} for more details) and the Malcev Yang-Baxter equation (MYBE)  given in \textnormal{\cite{postMalcev}}.
If $(A, \circ, [\cdot,\cdot])$ be a Malcev-Poisson algebra, then $r\in A\otimes A$ is said to be
skew-symmetric if
\begin{equation}
r=\sum\limits_i x_i\otimes y_i - y_i \otimes x_i. 
\end{equation}
For $r=\sum\limits_i x_i\otimes y_i\in A^{\otimes 2}$, denote 
\begin{equation}
\sigma(r)=\sum\limits_i  y_i \otimes x_i.
\end{equation}
 For a finite-dimensional linear space $V$ over $\mathbb{K}$, $V^*$
refers to the dual space of $V$
and then  $r$ can be regarded as a linear map from  the dual space $V^*$ to $V$ in the following way, for all $a^*,b^*\in V^*$,
\begin{equation}
\langle a^*, r(b^*)\rangle =\langle a^*\otimes b^*,r\rangle, \label{eq:idenrmap}
\end{equation}
where $\langle\cdot , \cdot\rangle$ is the canonical pairing between $V^{*}$ and $V$.

 Let $(V^*, \varrho^*, -\mu^*)$ be the dual representation
of the representation $(V, \varrho, \mu)$ of the Malcev-Poisson
algebra $A$.  Then there is a close relation between the $\mathcal{O}$-operator associated to $(V, \varrho, \mu)$ and the
(skew-symmetric) solutions of MPYBE in $A\ltimes_{(\varrho^*, -\mu^*)}V^*$. To articulate this result, we write $\mathcal{S}(A)$ for the set of all solutions of the
MPYBE on $A$ and denote by $\mathcal{O}_{A}(V, \varrho, \mu)$ the set of all $\mathcal{O}$-operators associated to the representation
$\varrho, \mu: A \to End(V ).$

Any linear map $T : V\to A$  can be identified as an
element in $A\otimes V^*\subset (A\ltimes_{(\varrho^*, -\mu^*)}V^*)\otimes (A\ltimes_{(\varrho^*, -\mu^*)}V)
$ as follows. Let $\{e_1,\cdots,e_n\}$ be a basis of $A$. Let
$\{v_1,\cdots, v_m\}$ be a basis of $V$ and $\{ v_1^*,\cdots,
v_m^*\}$ be its dual basis, that is $v_i^*(v_j)=\delta_{ij}$. We set
$$T(v_i)=\sum\limits_{j=1}^na_{ij}e_j, i=1,\cdots, m.$$ Since as
linear spaces, ${\rm Hom}(V,A)\cong A\otimes V^*$, we
have
$$T=\sum_{i=1}^m T(v_i)\otimes v_i^*=\sum_{i=1}^m\sum_{j=1}^n
a_{ij}e_j\otimes v_i^*\in A\otimes V^*\subset (A\ltimes_{(\varrho^*, -\mu^*)}V^*)\otimes (A\ltimes_{(\varrho^*, -\mu^*)}V^*).$$
\begin{thm}
With the above notations, $r=T-\sigma(T)\in \mathcal{S}(A\ltimes_{(\varrho^*, -\mu^*)}V^*)$   if and only if  $T\in \mathcal{O}_{A}(V, \varrho, \mu).$ 
\end{thm}
\begin{ex}
Let $ ad_A^*, -L^*:A\rightarrow gl(A^*)$ be the coadjoint representation 
of the Malcev-Poisson
algebra $A$. Then $r\in \mathcal{S}(A)$  if and only if $r\in \mathcal{O}_{A}(A^*, ad_A^*, -L^*).$
In fact,  $V = A^{*}$ and $V^{*} = A$. For any
linear map $T : A^{*}\to A$, $T$ can be identified as an element in $A\otimes A$ by
\begin{equation*}
\langle T(x^{*}), y^{*}\rangle = \langle x^{*}\otimes y^{*}, T\rangle, \forall x^{*}, y^{*} \in A^*.
\end{equation*}
Therefore, although $r=T-\sigma(T)\in (A\ltimes_{(ad, -L)}A)\otimes (A\ltimes_{(ad, -L)}A)$, in fact, $r \in A \otimes A$.
\end{ex}
\begin{ex}
Let $ ad, L:A\rightarrow gl(A)$ be the adjoint representation of the Malcev-Poisson
algebra $A$. We suppose that the Malcev-Poisson algebra
$A$ is equipped with a nondegenerate invariant symmetric bilinear form $B(, )$, that is,
$$B(x, y) =B(y, x), B([x, y], z)=B(x, [y, z]), B(x\circ y, z)=B(x, y\circ z). $$ 
Hence, $A^{*}$ is identified with $A$ since it is finite-dimensional. Let $r \in A \otimes A$. Then $r$ can be identified as a linear map from $A$ to
$A$. Then $r\in \mathcal{S}(A)$  if and only if $r\varphi \in \mathcal{O}_{A}(A, ad, L),$
where $\varphi: A\to A$  is a linear
map defined by $\langle \varphi(x),y\rangle = B(x, y)$ for any $x, y\in A.$
\end{ex}

\subsection{ Relative Rota-Baxter operators  and post-Malcev-Poisson algebras}

\begin{defn}[\!\!\cite{BaiGuoNi}]  Let $(A, \circ)$ and $(V, \circ_{V})$ be two commutative associative algebras.
Assume that $\mu: A\rightarrow End(V)$ is a  linear map such that  $(V,\mu)$ is a representation of  $(A,\circ)$. Then, $(V, \circ_{V}, \mu)$ is called an $A$-module associative algebra if
\begin{eqnarray}\label{modulecomass}
\mu(x)(a\circ_{V} b)=(\mu(x)a)\circ_{V} b, \forall\ x,y\in A,\ a,b\in V.
\end{eqnarray}
\end{defn}
\begin{defn}[\!\!\cite{postMalcev}]
Let $(A,[\cdot,\cdot])$ and $(V,[\cdot,\cdot]_V)$ be two Malcev algebras. 
Let
 $\varrho:A  \to End(V)$ be a linear map such that  $(V,\varrho)$ is a representation of  $(A,[\cdot,\cdot])$  and the following compatibility conditions hold for all $x, y, \in A,\ a,b, c \in V:$
\begin{align}\label{repKalgebras1}
&\varrho([x, y])[a,b]_V =\varrho(x)[\varrho(y)a, b]_V- [\varrho(y)\varrho(x)a, b]_V - [\varrho(x)\varrho(y)b, a]_V + \varrho(y)[\varrho(x)b, a]_V ,\\
\label{repKalgebras2} &[\varrho(x)a, \varrho(y)b]_V = [\varrho([x, y])a, b]_V - \varrho(x)[\varrho(y)a, b]_V + \varrho(y)\varrho(x)[a, b]_V + [\varrho(y)\varrho(x)b, a]_V,\\
\label{repKalgebras3} 
&[\varrho(x)a, [b, c ]_V ]_V= [[\varrho(x)b, a]_V, c]_V - \varrho(x)[[b, a]_V,c ]_V  - [\varrho(x)[a, c]_V , b]_V \\ 
&\hspace{9cm} - [[\varrho(x)c, b]_V, a]_V. \nonumber 
\end{align}
Then $(V,[\cdot,\cdot]_V,\varrho)$ is called an $A$-module Malcev algebra.
\end{defn}

\begin{defn}
Let $(A, [\cdot, \cdot], \circ)$ and $(V, [\cdot,\cdot]_V, \circ_{V} )$ be two Malcev-Poisson algebras.
Let  $\varrho, \mu: A\to End(V)$ are two  linear maps such that
\begin{enumerate}[label=\upshape{\arabic*)}, ref=\upshape{\arabic*}, labelindent=5pt, leftmargin=*]
  \item $ (V, \{\cdot,\cdot\}_V, \varrho)$ is an $A$-module Malcev algebras.
  \item $(V, \circ_{V}, \mu)$ is an $A$-module associative algebras.
  \item  $(V, \mu, \varrho)$ is a representation of $A$.
  \item  For any $x,y\in A$ and $a,b\in V$, the following equations hold:
\begin{eqnarray}
&&\label{modulemalpoi1}\varrho(x)(a\circ_{V} b)=(\varrho(x)a)\circ_{V} b + a\circ_{V} (\varrho(x)b),
\\
&& \label{modulemalpoi3}[a, \mu(x)b]_V = -(\varrho(x)a))\circ_{V} b + \mu(y)[a, b]_V.
\end{eqnarray}
\end{enumerate}
Then $(V, \circ_{V}, [\cdot,\cdot]_V, \mu, \varrho)$ is called an $A$-module Malcev-Poisson algebras.
\end{defn}
\begin{prop}
With notations as above, $(V, \circ_{V}, [\cdot,\cdot]_V, \mu, \varrho)$
is an $A$-module Malcev-Poisson algebra if and only if the direct sum of linear spaces $A\oplus V$
is turned into a Malcev-Poisson algebra by defining the operations for any $x,y\in A$ and $a,b\in V$ by 
\begin{eqnarray*}
&&\{x+a, y+b\}=[x,y]+ \varrho(x)b- \varrho(y)a+[a,b]_V,\\
&&(x+a)\bullet (y+b)=x\circ y+ \mu(x)b+\mu(y)a+ a\circ_{V} b.
\end{eqnarray*}
This is called the semi-direct product, often denoted by $A\ltimes V$.
\end{prop}
\begin{proof}

We only need to prove that \eqref{eq:malpoisalg} holds.
In fact, on one hand, 
\begin{eqnarray*}
&&\{x+ a,(y+b) \bullet (z+c)\}\\
&=&\{x+a, y\circ z+ \mu(y)c+\mu(z)b+ b\circ_{V} c\}\\
&=&[x, y\circ z]+ \varrho(x)\mu(y)c+ \varrho(x)\mu(z)b+ \varrho(x)(b\circ_V c)\\
&&- \varrho(y\circ z)a+[a,\mu(y)c]_V+[a,\mu(z)b]_V +[a,b\circ_V c]_V.
\end{eqnarray*}
On the other hand, 
\begin{eqnarray*}
&&\{x+a, y+b\}\bullet (z+c)+ (y+b)\bullet \{x+a, z+c\}\\
&=&([x,y]+ \varrho(x)b- \varrho(y)a+[a,b]_V)\bullet (z+c)\\
&&+ (y+b)\bullet([x,z]+ \varrho(x)c- \varrho(z)a+[a,c]_V)\\
&=&[x,y]\circ z+ \mu([x,y])c+\mu(z)\varrho(x)b-\mu(z)\varrho(y) a\\
&&+\mu(z)[a,b]_V+ (\varrho(x)b)\circ_{V} c - (\varrho(y)a)\circ_{V} c + ([a, b])\circ_{V} c\\
  &&+ y\circ [x,z]  +\mu(y)\varrho(x)c-\mu(y)\varrho(z) a+\mu(y)[a,c]_V\\
  &&+ \mu([x,z])b+ b\circ_{V} (\varrho(x)c)- b\circ_{V} (\varrho(z)a)+ b\circ_{V} ([a,c]).
\end{eqnarray*}
Thus, $(A\ltimes V,\bullet,\{\cdot,\cdot\})$ is a Malcev-Poisson algebra  if and only if $(V, \circ_{V}, [\cdot,\cdot]_V, \mu, \varrho)$  is an $A$-module Malcev-Poisson algebra.
    
\end{proof}
\begin{ex}
Let $(A, [\cdot, \cdot], \circ)$  be a Malcev-Poisson algebra. Then, $(A, \mathrm{ad}_{[,]}, L_\circ)$ is an $A$-module Malcev-Poisson algebra.
\end{ex}
\begin{defn}[\!\!\cite{Loday01}]
A  $\textbf{commutative dendriform trialgebra}$ is a $\mathbb{K}$-linear space $A$
equipped with two  bilinear operations $(\cdot, \vartriangleright)$ where $(A,\cdot)$ is a commutative associative algebras satisfying the following equations, for any $x,y,z\in A$, 
%\begin{eqnarray}
%&&x \cdot y=y\cdot x,
%\end{eqnarray}
%\begin{eqnarray}
%&&(x\cdot y)\cdot z=x\cdot (y\cdot z),
%\end{eqnarray}
\begin{eqnarray}
&&(x\circ y)\vartriangleright z=x\vartriangleright (y\vartriangleright z),
\\
&& (x\vartriangleright y)\cdot z=x\vartriangleright (y\cdot z).
\end{eqnarray}
\end{defn}
 \begin{defn}[\hspace{-0.1pt}\cite{postMalcev}]
 \label{def:postmalcev}
A $\textbf{post-Malcev algebra}$ $(A,[\cdot,\cdot],\diamond)$ is a Malcev algebra $(A,[\cdot,\cdot])$
together with a bilinear map $\diamond : A\otimes A \to A$ such that  for all $x,y,z \in A$,
\begin{align}
\label{postmalcevalgebras1}
&\{x, z\}\diamond [y, t] = x\diamond[z\diamond y, t]- [z\diamond(x\diamond y), t] - [x\diamond(z\diamond t), y] + z\diamond [x\diamond t, y],
\\
\label{postmalcevalgebras2}
&[x\diamond z, y\diamond t]= [\{x, y\}\diamond z, t] -  x\diamond[y\diamond z, t] +  y\diamond(x\diamond[z, t]) + [y\diamond(x\diamond t), z], \\
\label{postmalcevalgebras3}
&[x\diamond z, [y, t]] = [[x\diamond y, z], t] - x\diamond [[y, z], t] - [x\diamond[z, t], y] - [[x\diamond t, y], z],
\\
\label{postmalcevalgebras4}
&\{y,z\}\diamond(x\diamond t)=\{\{x,y\},z\}\diamond t+y\diamond(\{x,z\}\diamond t)+x\diamond(y\diamond( z\diamond t))-z\diamond(x\diamond( y\diamond t)).
\end{align}
\end{defn}
\begin{prop}\label{postmalcev==>malcev}
Let $(A,[\cdot,\cdot],\diamond)$ be a post-Malcev algebra. 
Then the bracket
\begin{equation}\label{malcevbracket}
      \{x,y\}=x\diamond y- y\diamond x+ [x,y]
\end{equation}
defines a Malcev algebra structure on $A$. We denote this algebra by $A^C$ and we call it the sub-adjacent Malcev algebra of $A$. Furthermore, $\mathcal{L}_{\diamond} : A\to A$ defined for all $x, y\in  A$ by
\begin{equation}
 \mathcal{L}_{\diamond}(x)y = x\diamond y, 
 \end{equation}
gives an $A$-module Malcev algebra of $A^C$ on $A$.
\end{prop}
\begin{defn}
A \textbf{post-Malcev-Poisson algebra} is tuple $(A,[\cdot,\cdot],\diamond,\cdot,\vartriangleright)$, where  $(A,[\cdot,\cdot],\cdot)$  is a Malcev-Poisson algebra,  $(A,[\cdot,\cdot],\diamond)$  is a post-Malcev algebra and  $(A,\cdot,\vartriangleright)$ is a commutative dendriform trialgebra satisfying the following compatible conditions, for any $x,y,z\in A$,
\begin{equation}\label{eq:postmalcpoi1}
x \diamond (y \cdot z) = (x \diamond y) \cdot z + y \cdot (x \diamond z),
\end{equation}
 \begin{equation}\label{eq:postmalcpoi2}
 [x, z\vartriangleright y] = z\vartriangleright [x, y] - y \cdot (z \diamond x),
\end{equation}
\begin{equation}\label{eq:postmalcpoi3}
( y \circ  z) \diamond x = z\vartriangleright(y \diamond x) + y\vartriangleright(z \diamond x),
\end{equation}
\begin{equation}\label{eq:postmalcpoi4}
\{x, z\} \vartriangleright y = x \diamond (z\vartriangleright y)-z\vartriangleright(x \diamond y),
 \end{equation}
 where the operation $\circ$ and bracket $\{\cdot,\cdot\}$ are defined by
\begin{equation}\label{eq:postmanifold-operations}
  x\circ y=x\vartriangleright y+y\vartriangleright x + x\cdot y,\quad \{x,y\}=x\diamond y-y\diamond x + [x, y].
\end{equation}
\end{defn}
\begin{ex}
\begin{enumerate}[label=\upshape{\arabic*)}, ref=\upshape{\arabic*}, labelindent=5pt, leftmargin=*]
  \item As Malcev algebras are a generalization of Lie algebras, the post-Malcev-Poisson algebras are a generalization of post-Poisson algebras \textup{\cite{NiBai}}.
  \item Let $(A, [\cdot,\cdot], \diamond, \cdot, \vartriangleright)$ be a post-Malcev-Poisson algebra. If the operations $[\cdot, \cdot]$ and $\cdot$ are trivial, then it
is a pre-Malcev-Poisson algebra introduced in \cite{Fatma} in the study of relative Rota-Baxter operators of weight zero on Malcev-Poisson algebras (see \cite{Sami}) which is a generalization of pre-Poisson algebra given by M. Aguiar  in \cite{Aguiar00}.
\end{enumerate}
\end{ex}

We get the following conclusion.

\begin{thm}
The triple  $(A,\circ,\{\cdot,\cdot\})$ is a Malcev-Poisson algebra, 
where the operation $\circ$ and bracket 
$\{\cdot,\cdot\}$ are given by \eqref{eq:postmanifold-operations}, which is called the {\textbf associated
Malcev-Poisson algebra} of
$(A, [\cdot,\cdot], \diamond, \cdot, \vartriangleright)$ and is denoted by $A^C$.
\end{thm}
\begin{proof}
According to Proposition \ref{postmalcev==>malcev}, we have that $(A,\circ)$ is a commutative associative algebra, and $(A,\{\cdot,\cdot\})$ is a Malcev algebra.
It remains to show that identity \eqref{eq:malpoisalg} holds. In fact, 
\begin{eqnarray*}
   \{x,y \circ z\}&=& \{x, y\vartriangleright z+z\vartriangleright y + y\cdot z\}\\
  &=& x\diamond (y\vartriangleright z)+ x\diamond (z\vartriangleright y)+ x\diamond (y\cdot z)- (y\vartriangleright z)\diamond x- (z\vartriangleright y)\diamond x - (y\cdot z)\diamond x\\
  &&+[x,y\vartriangleright z]+[x,z\vartriangleright y]+[x,y\cdot z]\\
   &=& y \vartriangleright(x\diamond z) + \{x,y\}\vartriangleright z+z \vartriangleright(x\diamond y) + \{x,z\}\vartriangleright y\\
   &&+(x\diamond y)\cdot z+y\cdot(x\diamond z)-z\vartriangleright(y\diamond x)-y\vartriangleright(z\diamond x)\\
   &&+y\vartriangleright[x,z]-z\cdot(y\diamond x)+z\vartriangleright [x,y]-y\cdot(z\diamond x)+[x,y]\cdot z+y\cdot[x,z]\\
   &=&y\vartriangleright\big( x\diamond z-z\diamond x+[x,z]\big)+ \{x,z\}\vartriangleright y\\
   &&+y\cdot\big(x\diamond z-z\diamond x+[x,z]\big)+\{x,y\}\vartriangleright z\\
   &&+ z\vartriangleright\big(x\diamond y-y\diamond x+[x,y]\big)+\big(x\diamond y-y\diamond x+[x,y]\big)\cdot z\\
   &=&y\vartriangleright\{x,z\}+\{x,z\}\vartriangleright y+y\cdot\{x,z\}+\{x,y\}\vartriangleright z+z\vartriangleright \{x,y\}+\{x,y\}\cdot z\\
   &=& y\circ \{x,z\}+ \{x,y\}\circ z.
\end{eqnarray*}
Thus $(A,\circ,\{\cdot,\cdot\})$ is a Malcev-Poisson algebra.
\end{proof}

\begin{prop}
Let  $(A,[\cdot,\cdot],\diamond,\cdot,\vartriangleright)$ be a post-Malcev-Poisson algebra. Define
$\mathcal{L}_\diamond, \mathcal{L}_\vartriangleright:A\to A$ by $\mathcal{L}_{\diamond}(x)y=x\diamond y$ and $\mathcal{L}_{\vartriangleright}(x)y=x\vartriangleright y$ for any $x,y\in A$.
Then,
$(A, [\cdot,\cdot], \cdot, \mathcal{L}_\diamond, \mathcal{L}_\vartriangleright)$ is an $A$-module Malcev-Poisson algebra of  $(A^{C},\circ,\{\cdot,\cdot\})$.
\end{prop}
\begin{proof}
By Proposition \ref{postmalcev==>malcev}, $(A; [\cdot,\cdot], \mathcal{L}_\diamond)$ is an $A$-module Malcev algebra of the sub-adjacent Malcev algebra $A^{C}$. Moreover,  \eqref{eq:postmalcpoi1} implies that   \eqref{modulemalpoi1} holds, and \eqref{eq:postmalcpoi2} implies that \eqref{modulemalpoi3} holds. Thus $(A, [\cdot,\cdot],\cdot, \mathcal{L}_\diamond, \mathcal{L}_\vartriangleright)$ is an $A$-module Malcev-Poisson algebra of  $(A^{C},\circ,\{\cdot,\cdot\})$.
%\begin{align*}
%&\mathcal{L}_\diamond(x)(y\cdot z)-(\mathcal{L}_\diamond(x)y)\cdot z - y\cdot(\mathcal{L}_\diamond(x)z)\\
%&=x\diamond (y\cdot z)-(x\diamondy)\cdot z - y\cdot(x\diamond z)
%\end{align*}
\end{proof}
Now, we introduce the notion of weighted relative Rota-Baxter operators of weight $\lambda\in \mathbb{K}$ of Malcev-Poisson algebras generalizing the notion of weighted Rota-Baxter operators of weight $\lambda\in \mathbb{K}$.
\begin{defn}
Let $(A,  \circ)$ be a commutative associative algebra and $( V, \circ_{V}, \mu)$ an $A$-module associative algebras.
A linear map $\mathcal{R}:V\rightarrow A$ is called a {\bf weighted relative Rota-Baxter operator} of weight $\lambda$ on a commutative associative algebra $(A,\circ)$ with respect to an $A$-module associative algebras if $\mathcal{R}$ satisfies, for all $a,b\in V$,
  \begin{equation}
  \label{o-opcomass}\mathcal{R}(a)\circ \mathcal{R}(b)=\mathcal{R}(\mu(\mathcal{R}(a))b+\mu(\mathcal{R}(b))a+\lambda a\circ_{V} b).
  \end{equation}
  \end{defn}
  \begin{ex}
  When $(V, \circ_{V}, \mu)=(A, \circ, L_{\circ})$,    \eqref{o-opcomass} becomes
\begin{eqnarray}
&&\label{RB1}\mathcal{R}(x)\circ\mathcal{R}(y)=\mathcal{R}(\mathcal{R}(x)\circ y+x \circ \mathcal{R}(y)+\lambda x\circ y),
\end{eqnarray}
and \eqref{RB1}  implies that $\mathcal{R}: A\to A$ is a weighted Rota-Baxter operator of weight $\lambda\in \mathbb{K}$ on the commutative associative algebra $(A, \circ)$.
  \end{ex}
  \begin{lem}\label{lem:comasstocommdendtri}
Let $(A, \circ)$ be a commutative associative   algebra and $(V, \circ_{V}, \mu)$ its $A$-module algebra.
Let  $\mathcal{R}: V\to A$  be a weighted relative Rota-Baxter operator of weight $\lambda\in \mathbb{K}$ on $A$ associated to
$(V, \circ_{V}, \mu)$. Then, there exists a commutative dendriform trialgebra algebra structure on $V$ given  
for all $a, b\in V$ by
\begin{eqnarray}\label{comass==>comden}
 a\cdot b=\lambda a \circ_{V} b, \quad a\vartriangleright b=\mu(\mathcal{R}(a))b.
\end{eqnarray}
\end{lem}
\begin{defn}
  Let $(A, [\cdot,\cdot])$ be a Malcev algebra and $(V, [\cdot,\cdot]_V, \varrho)$ an $A$-module Malcev algebra.
 A linear map $\mathcal{R}:V\rightarrow A$ is called a {\bf weighted relative Rota-Baxter operator} of weight $\lambda$ on a Malcev algebra $(A, [\cdot,\cdot])$ with respect to an $A$-module Malcev   algebra $(V, [\cdot,\cdot]_V, \varrho)$  if , for any $a,b\in V$,
\begin{equation}
\label{o-oplie}[\mathcal{R}(a), \mathcal{R}(b)]=\mathcal{R}(\varrho(\mathcal{R}(a))b-\varrho(\mathcal{R}(b))a+\lambda[a,b]_V).
\end{equation}
\end{defn}
\begin{ex}
When $(V, [\cdot,\cdot]_V, \varrho)=(A, [\cdot,\cdot], ad_{[,]})$,     \eqref{o-oplie} becomes
\begin{eqnarray}
&&\label{RB2}[\mathcal{R}(x), \mathcal{R}(y)]=\mathcal{R}([\mathcal{R}(x), y]+[x, \mathcal{R}(y)]+\lambda[x,y]),
\end{eqnarray}
and  \eqref{RB2} implies that $\mathcal{R}: A\to A$ is a weighted Rota-Baxter operator of weight $\lambda\in \mathbb{K}$ on the Malcev algebra $(A, [\cdot,\cdot])$.
  \end{ex}
  \begin{lem}\label{lem:liealgtopostliealg}
Let $(A, [\cdot,\cdot])$ be a Malcev  algebra and $(V, [\cdot,\cdot]_V, \varrho)$ its $A$-module  algebra.
Let  $\mathcal{R}: V\to A$  be a weighted relative Rota-Baxter operator of weight $\lambda\in \mathbb{K}$ on $A$ associated to
$(V, [\cdot,\cdot]_V, \varrho)$. Then there exists a post-Malcev algebra structure on $V$ given for all $a, b\in V$ by
\begin{eqnarray}\label{lie==>Postlie}
\{a,b\}=\lambda[a,b]_V, \quad a\diamond b=\varrho(\mathcal{R}(a))b.
\end{eqnarray}
\end{lem}
  \begin{defn}
Let $(A, [\cdot,\cdot], \circ)$ be a Malcev-Poisson algebra and $(V, [\cdot,\cdot], \circ, \varrho, \mu)$ an $A$-module Malcev-Poisson algebra.
 A linear map $\mathcal{R}: V\to A$ is called a weighted relative Rota-Baxter operator of weight $\lambda\in \mathbb{K}$ associated to
$(V, [\cdot,\cdot], \circ, \varrho, \mu)$ if  $\mathcal{R}$ is both
a weighted relative Rota-Baxter operator of weight $\lambda\in \mathbb{K}$ on the commutative associative algebra $(A,\circ)$ and a weighted relative Rota-Baxter operator of weight $\lambda\in \mathbb{K}$  on the Lie algebra
$(A,[\cdot,\cdot])$.
\end{defn}
\begin{ex}
 A weighted relative Rota-Baxter operator of weight $\lambda\in \mathbb{K}$ on a Malcev-Poisson  algebra associated to
$(A, [\cdot,\cdot], \circ, \mathrm{ad}_{[,]}, L_{\circ})$ is called a  weighted Rota-Baxter operator of weight $\lambda\in \mathbb{K}$ on the Malcev-Poisson algebra $(A, [\cdot,\cdot], \circ)$.
%When $(B, \{\cdot,\cdot\}_B, \c_{B}, \varrho, \mu)=(A, \{\cdot,\cdot\}, \c_{A}, ad_{\{\cdot,\cdot\}_A}, L_{\c_{A}})$,   \eqref{o-opcomass} and \eqref{o-oplie} become
%\begin{eqnarray}
%&&\label{RB1}\mathcal{R}(x)\c_{A} \mathcal{R}(y)=\mathcal{R}(\mathcal{R}(x)\c_{A} y+x \c_{A} \mathcal{R}(y)+\lambda x\c_{A} y),\\
%&&\label{RB2}\{\mathcal{R}(x), \mathcal{R}(y)\}_A=\mathcal{R}(\{\mathcal{R}(x), y\}_A+\{x, \mathcal{R}(y)\}_A+\lambda\{x,y\}_A),
%\end{eqnarray}
%respectively.   \eqref{RB1} and \eqref{RB2} imply that $\mathcal{R}: A\varrhoightarrow A$ is a Rota-Baxter operator of weight $\lambda\in \mathbb{K}$ on the commutative associative algebra $(A, \c_{A})$
% and on the Lie algebra $(A, \{\cdot,\cdot\}_A)$  respectively.
 \end{ex}
 \begin{thm}\label{thm:malpoitopostmalpoi}
Let $(A, \circ,[\cdot,\cdot])$ be a Malcev-Poisson  algebra and $(V, [\cdot,\cdot], \circ, \varrho, \mu)$ be an $A$-module Malcev-Poisson algebra.
Let  $\mathcal{R}: V\to A$  be a weighted relative Rota-Baxter operator of weight $\lambda\in \mathbb{K}$ associated to
$(V, [\cdot,\cdot], \circ, \varrho, \mu)$.
Define new  operations $[\cdot,\cdot], \diamond, \cdot, \vartriangleright: V\times V\to V$ for any $a, b\in V$ by
\begin{eqnarray}\label{malcev==>Postmalcev}
\{a,b\}=\lambda[a,b]_V,\quad a\diamond b=\varrho(\mathcal{R}(a))b,\quad a\cdot b=\lambda a \circ_{V} b,\quad a\vartriangleright b=\mu(\mathcal{R}(a))b.
\end{eqnarray}
Then $(V, \{\cdot,\cdot\}, \diamond, \cdot, \vartriangleright)$ is a post-Malcev-Poisson algebra and $\mathcal{R}$ is a homomorphism of Malcev-Poisson algebras
from the associated Malcev-Poisson algebra $V ^{c}$ of $(V, \{\cdot,\cdot\}, \diamond, \cdot, \vartriangleright)$ to $(A, [\cdot,\cdot], \circ)$.
\end{thm}
\begin{proof}
By Lemma \ref{lem:comasstocommdendtri} and Lemma \ref{lem:liealgtopostliealg}, we know that $(V, \{\cdot,\cdot\}, \diamond)$ is a post-Malcev algebra and $(V, \cdot, \vartriangleright)$ is a commutative dendriform trialgebra.

Furthermore, for any $a,b,c\in V$, by   \eqref{modulemalpoi1} and \eqref{modulemalpoi3}, 
\begin{eqnarray*}
  a\diamond(b\cdot c)&=& \lambda\varrho(\mathcal{R}(a))(b\circ_V c) \\
   &=& \lambda( \varrho(\mathcal{R}(a))b)\circ_V c+\lambda b\circ_V(\varrho(\mathcal{R}(a))c) \\
   &=& (a\diamond b)\cdot c+ b\cdot (a\diamond c),
\\
  \{a, c\vartriangleright b\}&=& \lambda [a,\mu(\mathcal{R}(c))b]_V \\
   &=& -\lambda( \varrho(\mathcal{R}(c))a)\circ_V b+\lambda\mu(\mathcal{R}(c))[a,b]_V \\
   &=& -(c\diamond a)\cdot b+ c\vartriangleright  \{a, b\}.
\end{eqnarray*}
This implies that   \eqref{eq:postmalcpoi1} and \eqref{eq:postmalcpoi2} hold.
Similarly, by   \eqref{eq:repmalcpoisalg1},  \eqref{eq:postmanifold-operations} and \eqref{o-opcomass}, we can verify that  \eqref{eq:postmalcpoi3} holds. In fact
\begin{eqnarray*}
  (b\circ c)\diamond a &=& (b\vartriangleright c+c\vartriangleright b+b\cdot c)\diamond a  \\
   &=& \varrho\big(\mathcal{R}(\mu(\mathcal{R}(b))c + \mu(\mathcal{R}(c))b+ \lambda b\circ_V c)\big)a\\
   &=&  \varrho(\mathcal{R}( b)\circ \mathcal{R}( c))a\\
   &=& \mu(\mathcal{R}( c))\varrho(\mathcal{R}( b))a +\mu(\mathcal{R}( b))\varrho(\mathcal{R}( c))a\\
   &=& c\vartriangleright(b\diamond a)+ b\vartriangleright (c\diamond a).
\end{eqnarray*}
Moreover, by \eqref{eq:repmalcpoisalg2},  \eqref{eq:postmanifold-operations} and \eqref{o-oplie}, 
\begin{eqnarray*}
\{a, c\}_{\mathcal{R}} \vartriangleright b &=&(a\diamond c- c\diamond a+ [a,c]) \vartriangleright b\\
&=& \mu\big(\mathcal{R}( \varrho(\mathcal{R}(a))c-\varrho(\mathcal{R}(c))a+ \lambda [a,c]_V)\big)b\\
&=&\mu([\mathcal{R}(a),\mathcal{R}(c)])b\\
&=&\varrho(\mathcal{R}(a))\mu(\mathcal{R}(c))b-\mu(\mathcal{R}(c))\varrho(\mathcal{R}(a))b\\
&=&a \diamond (c\vartriangleright b)-c\vartriangleright(a \diamond b),
\end{eqnarray*}
which implies that \eqref{eq:postmalcpoi4} holds. Then  $(V, \{\cdot,\cdot\}, \diamond, \cdot, \vartriangleright)$ is a post-Malcev-Poisson algebra and $\mathcal{R}$ is a homomorphism of Malcev-Poisson algebras
from the associated Malcev-Poisson algebra $V ^{c}$ of $(V, \{\cdot,\cdot\}, \diamond, \cdot, \vartriangleright)$ to $(A, \circ, [\cdot,\cdot])$.
\end{proof}
\begin{ex}
 Let $(A, \circ,[\cdot,\cdot])$ be a Malcev-Poisson algebra and $\mathcal{R}: A\to A$  be a weighted  Rota-Baxter operator of weight $\lambda\in \mathbb{K}$.
Define  new operations $\{\cdot,\cdot\}, \diamond, \cdot, \vartriangleright: A\times A\to A$ for any $x, y\in A$ as 
\begin{eqnarray}\label{malcevpoistoPostmalpoi}
\{x,y\}=\lambda[x,y],\quad x\diamond y=[\mathcal{R}(x), y],\quad x\cdot y=\lambda x \circ y,\quad x\vartriangleright y=\mathcal{R}(a)\circ b.
\end{eqnarray}
 Then $(A, \{\cdot,\cdot\}, \diamond, \cdot, \vartriangleright)$ is a post-Malcev-Poisson algebra and $\mathcal{R}$ is a homomorphism of Malcev-Poisson algebras
from the associated Malcev-Poisson algebra $A ^{c}$ of $(A, \{\cdot,\cdot\}, \diamond, \cdot, \vartriangleright)$ to $(A, \circ, [\cdot,\cdot])$.
\end{ex}
%%%%%%%%%%%%%%%%%%%%%%%%%%%%%%%%%%%%%%%%%%
%%%%%%%%%%%%%%%%%%%%%%%%%%%%%%%%%%%%%%%%%%
%%%%%%%%%%%%%%%%%%%%%%%%%%%%%%%%%%%%%%%%%%

\end{document}